\documentclass[12pt]{article}
\usepackage{graphicx}
\usepackage{amsmath}
\usepackage{amsfonts}
\usepackage{amssymb}
\usepackage{amsthm}
\usepackage[shortlabels]{enumitem}
\usepackage{multicol}
\usepackage[utf8]{inputenc}
\usepackage[a4paper, left=2.5cm, right=2.5cm, top=3cm, bottom=3cm, headsep=1.2cm]{geometry}

\usepackage{caption}
\usepackage{color}
\definecolor{sepia}{cmyk}{0, 0.83, 1, 0.70}
\definecolor{mahogany}{cmyk}{0   , 0.85, 0.87, 0.35}
\definecolor{olivegreen}    {cmyk}{0.64, 0   , 0.95, 0.40}
\usepackage[plainpages=false,pdfpagelabels,bookmarksnumbered,%
 colorlinks=true,%
 linkcolor=sepia,%
 citecolor=sepia,%
 unicode]{hyperref}

\newtheorem{theorem}{Theorem}

\newtheorem{corollary}[theorem]{Corollary}

\newtheorem{lemma}[theorem]{Lemma}
\newtheorem{proposition}[theorem]{Proposition}

\theoremstyle{definition}

\newtheorem{remark}[theorem]{Remark}

\numberwithin{equation}{section} 
\numberwithin{theorem}{section}  

\makeatletter
\renewenvironment{proof}[1][\proofname]
{\par
	\pushQED{$\blacksquare$} 
	\normalfont\topsep6\p@\@plus6\p@\relax
	\trivlist
	\item[\hskip\labelsep\bfseries#1\@addpunct{.}]
	\ignorespaces}
{\popQED \endtrivlist\@endpefalse}
\makeatother

\DeclareMathOperator*{\fix}{Fix}

\begin{document}

\title{\vspace{-3em} {\textbf{Polynomial Estimates for the Method of Cyclic Projections in Hilbert Spaces}}}

\author{Simeon Reich\thanks{Department of Mathematics,
    The Technion -- Israel Institute of Technology, 3200003 Haifa, Israel;
    E-mail: \texttt{sreich@technion.ac.il}. } \and
    Rafa\l\ Zalas\thanks{Department of Mathematics,
    The Technion -- Israel Institute of Technology, 3200003 Haifa, Israel;
    E-mail: \texttt{rafalz@technion.ac.il}.  }
}

\maketitle

\begin{abstract}
  We study the method of cyclic projections when applied to closed and linear subspaces $M_i$, $i=1,\ldots,m$, of a real Hilbert space $\mathcal H$. We show that the average distance to individual sets enjoys a polynomial behaviour $o(k^{-1/2})$ along the trajectory of the generated iterates. Surprisingly, when the starting points are chosen from the subspace $\sum_{i=1}^{m}M_i^\perp$, our result yields a polynomial rate of convergence $\mathcal O(k^{-1/2})$ for the method of cyclic projections itself. Moreover, if $\sum_{i=1}^{m} M_i^\perp$ is not closed, then both of the aforementioned rates are best possible in the sense that the corresponding polynomial $k^{1/2}$ cannot be replaced by $k^{1/2+\varepsilon}$ for any $\varepsilon >0$.


\end{abstract}

\section{Introduction}
Let $\mathcal H$ be a real Hilbert space with its inner product denoted by $\langle \cdot, \cdot \rangle$ and its induced norm denoted by $\|\cdot\|$. Throughout this paper we assume that for each $i = 1, \ldots, m$, the set $M_i$ is a closed and linear subspace of $\mathcal H$ and we put $M := \bigcap_{i=1}^m M_i$. We denote by $P_{M_i}$ and $P_M$, the orthogonal projections onto $M_i$ and $M$, respectively, $i = 1,\ldots,m.$ The method of cyclic projections for the subspaces $M_i$ is defined by
\begin{equation}\label{int:yk}
  y_0 \in \mathcal H, \quad y_{k} := (P_{M_m} \ldots P_{M_1})^k(y_0), \quad k = 1,2,\ldots,
\end{equation}
where in order to shorten the notation, we put
\begin{equation}\label{int:T}
  T := P_{M_m} \ldots P_{M_1}.
\end{equation}
Thanks to von Neumann \cite{Neumann1949} ($m=2$) and Halperin \cite{Halperin1962} ($m \geq 2$), we know that:
\begin{theorem}\label{thm:norm}
  For each $y_0\in\mathcal H$, we have $\|y_k-P_M(y_0)\| \to 0$ as $k \to \infty$.
\end{theorem}
Since then, the method of cyclic projections has been extensively studied in the literature; see, for example, \cite{BauschkeBorwein1996, Deutsch2001, Galantai2004, Cegielski2012, Popa2012}. In this paper we study the asymptotic properties of the error term
\begin{equation}\label{int:error}
  \|y_k-P_M(y_0)\|,
\end{equation}
the average distance to the individual sets
\begin{equation} \label{int:proxAVG}
  {\textstyle\sqrt{\frac 1 m \sum_{i=1}^m d^2(y_k,M_i)}}
\end{equation}
and the increment
\begin{equation} \label{int:proxAR}
  \|y_k - y_{k-1}\|
\end{equation}
along the trajectory $\{y_k\}_{k=0}^\infty$, assuming that the subspace $\sum_{i=1}^{m} M_i^\perp$ is not closed. Below we present a very brief overview of the relevant literature.

We begin with a result by Bauschke, Borwein and Lewis  \cite[Theorem 3.7.4]{BauschkeBorweinLewis1997} according to which:
\begin{theorem}\label{thm:linear}
  The subspace $\sum_{i=1}^{m} M_i^\perp$ is closed if and only if $\|T - P_M\| < 1$.
\end{theorem}

For historical developments concerning Theorem \ref{thm:linear} we refer the interested reader to \cite[p. 235]{Deutsch2001}. We note here briefly that for $m=2$, the subspace $M_1^\perp+M_2^\perp$ is closed $\Leftrightarrow \cos(M_1,M_2) < 1$, where  $\cos(M_1,M_2) := \sup\{\langle x_1,x_2\rangle \colon x_i \in M_i \cap (M_1\cap M_2)^\perp \text{ and } \|x_i\| \leq 1, i=1,2\}$ is the  cosine of the Friedrichs angle. In fact, for $m=2$ we have $\|T^k - P_{M}\| = \cos^{2k-1}(M_1,M_2)$; see \cite{Aronszajn1950, KayalarWeinert1988}. Interestingly enough, analogous formulas involving $\cos(M_1,M_2)$ have been established for other projection methods; see, for example, \cite{BauschkeCruzNghiaPhanWang2014, BauschkeCruzNghiaPhanWang2016, ReichZalas2017}  and \cite[Table 1]{ArtachoCampoy2019}.

Because of the inequality $\|T^k - P_M\| \leq \|T-P_M\|^k$, which holds for $m \geq 2$ (see \cite[Corollary 1]{KayalarWeinert1988}), the closedness of $\sum_{i=1}^{m}M_i^\perp$ implies linear rate of convergence for the error term \eqref{int:error}, that is,
\begin{equation}\label{int:linear:eq}
  \|y_k - P_M(y_0)\| = \mathcal O(q^k)
\end{equation}
for some $q \in (0,1)$. Moreover, the same linear rate of convergence $\mathcal O(q^k)$ holds for the average distance \eqref{int:proxAVG} and the  increment  \eqref{int:proxAR} as both of them can be bounded from above by $2\|y_{k-1}-P_M(y_0)\|$.

The question what happens with the error term \eqref{int:error} when the subspace $\sum_{i=1}^{m} M_i^\perp$ is not closed was answered much later by Bauschke, Deutsch and Hundal in \cite[Theorem 1.4]{BauschkeDeutschHundal2009} for $m=2$ and in \cite[Theorem 6.4]{DeutschHundal2010} for $m \geq 2$.

\begin{theorem}\label{thm:dichotomy}
Assume that $\sum_{i=1}^r M_i^\perp$ is not closed. Then for each $y_0 \in \mathcal H$, the sequence  $\{y_k\}_{k=1}^\infty$ converges in norm to $P_M(y_0)$, but the convergence is arbitrarily slow, that is, for any sequence $(a_k)_{k=0}^\infty$ of positive numbers converging to zero, there is $y_0 \in \mathcal H$ such that
\begin{equation}\label{}
  \|y_k-P_M(y_0)\| \geq a_k, \quad k=1,2,\ldots .
\end{equation}
\end{theorem}
The first example of two subspaces with the arbitrarily slow convergence phenomena was presented by Franchetti and Light in \cite{FranchettiLight1986}. The arbitrary slow convergence for the method of cyclic projections is also discussed in \cite{BadeaGrivauxMuller2011, DeutschHundal2011, DeutschHundal2015}. Interestingly, analogous results hold for other projection methods; see, for example, \cite{BadeaSeifert2017, BauschkeDeutschHundal2009, BorodinKopecka2020, ReichZalas2017, ReichZalas2021}. The alternative between linear and arbitrarily slow convergence is known as the dichotomy theorem.

Theorem \ref{thm:dichotomy} implies that if the subspace $\sum_{i=1}^r M_i^\perp$ is not closed, then there cannot be a polynomial upper bound $\mathcal O(k^{-p})$ for \eqref{int:error} that holds for  some $p > 0$ and  all $y_0 \in \mathcal H$. This, however, does not rule out the existence of such upper bounds for the increment \eqref{int:proxAR}. In fact:

\begin{theorem}\label{thm:ARMi}
  For each $y_0\in\mathcal H$, we have $\|y_k-y_{k-1}\| = o(k^{-1})$.
\end{theorem}

This result has been established by Badea and Seifert in \cite[Theorem 2.1 and Remark 4.2(b)]{BadeaSeifert2016} for the product of orthogonal projections in a complex Hilbert space. Since Theorem \ref{thm:ARMi} plays an important role in our analysis, we elaborate more on its proof for a real Hilbert space in the \hyperref[sec:Appendix]{Appendix}. We also present an alternative proof by using \cite[Lemma 5.2]{Crouzeix2008}. A similar result can be found in \cite[Proposition 2.2]{Cohen2007} for the product of conditional expectations.

 It turns out that in contrast to the arbitrarily slow convergence, we can still expect  polynomial behavior for the error term \eqref{int:error} if the starting points $y_0$ belong to a certain subspace of $\mathcal H$.  For example, in view of Theorem \ref{thm:ARMi}, one of the candidates is the subspace $M \oplus (I-T)(\mathcal H)$, which is dense in $\mathcal H$ and on which \eqref{int:error} converges with the rate $o(k^{-1})$. A more general result of Badea and Seifert \cite[Theorem 4.3]{BadeaSeifert2016} (see the \hyperref[sec:Appendix]{Appendix}) asserts that:
\begin{theorem}\label{thm:superPoly}
  If $\sum_{i=1}^r M_i^\perp$ is not closed, then for each  $y_0 \in X_p := M \oplus (I-T)^p(\mathcal H)$ (which is a dense linear subspace of $\mathcal H$), the convergence is polynomial and we have
  \begin{equation}\label{thm:superPoly:eq}
    \|y_k - P_M(y_0)\| = o(k^{-p}),
  \end{equation}
  where $p = 1,2,\ldots$. Moreover, for each $y_0 \in X := \bigcap_{p=1}^\infty X_p$ (which is also a dense linear subspace of $\mathcal H$), the convergence is super-polynomially fast as \eqref{thm:superPoly:eq} holds for all $p>0$.
\end{theorem}

The authors commented in \cite[Remark 2.5 (c)]{BadeaSeifert2016} that the rate in \eqref{thm:superPoly:eq} is optimal in the sense that it cannot be improved for all $y_0 \in X_p$. In particular, the inclusions $X_{p+1} \subset X_p$ are strict for all $p = 1,2,\ldots$. We elaborate further on this below. It is also worth mentioning that \cite[Theorem 4.3]{BadeaSeifert2016} allows real values of $p>0$ in \eqref{thm:superPoly:eq} for which the corresponding subspaces $X_p$ are defined by using the so-called fractional powers of operators.

Subsequently, Borodin and Kopeck\'{a} \cite[Theorems 3 and 4]{BorodinKopecka2020} managed to show two polynomial error bounds when the set of starting points is restricted to the subspace $\sum_{i=1}^{m} M_i^\perp$  and when $M = \{0\}$. We slightly rephrase their result allowing $M \neq \{0\}$.

\begin{theorem}\label{thm:polyRate}
  For each $y_0 \in  Y := M \oplus  \sum_{i=1}^{m} M_i^\perp$, we have
  \begin{equation}\label{thm:polyRate:msets}
    \|y_k -  P_M(y_0)  \| = \mathcal O(k^{ - 1/(4m\sqrt m + 2)  }).
  \end{equation}
  Moreover, if the number of subspaces $m = 2$, then for each $y_0 \in   Y = (M_1\cap M_2) \oplus  M_1^\perp + M_2^\perp$, we have
  \begin{equation}\label{thm:polyRate:2sets}
    \|y_k -  P_M(y_0)  \| = \mathcal O(k^{-1/2}).
  \end{equation}
  Furthermore, the rate in \eqref{thm:polyRate:2sets} is best possible as the corresponding polynomial $k^{1/2}$ cannot be replaced by $k^{1/2+\varepsilon}$ for any $\varepsilon > 0$.
\end{theorem}

Finding the best possible power $p >0$ for the upper bound $\mathcal O(k^{-p})$ in \eqref{thm:polyRate:msets} was left as an open problem when $m \geq 3$; see \cite[Problem 3]{BorodinKopecka2020}. It is worth emphasizing that the optimality of \eqref{thm:polyRate:2sets} is shown in \cite{BorodinKopecka2020} by  using  an example of two subspaces of a separable Hilbert space $\mathcal H$ for which $M_1^\perp + M_2^\perp$ is not closed. In that particular example, for each $\varepsilon >0 $, the authors define an $\varepsilon$-dependent starting point $y_0 \in M_1^\perp + M_2^\perp$ for which the sequence $\{y_k\}_{k=0}^\infty$ satisfies the lower bound
\begin{equation}\label{int:lowerBound}
  \|y_k\| \geq C k^{-1/2-\varepsilon}, \quad k=1,2,\ldots,
\end{equation}
where $C = C(y_0) > 0$. The example extends to the case where $m \geq 3$ by simply putting $M_i:=M_2$ for $i \geq 3$. In particular, the power $p>0$ in \eqref{thm:polyRate:msets} cannot be larger than $1/2$. We return to \cite[Problem 3]{BorodinKopecka2020} and the lower bound property \eqref{int:lowerBound} below.

 Similarly to the subspaces $X_p$ considered in Theorem \ref{thm:superPoly}, the subspace $Y$ defined in Theorem \ref{thm:polyRate} is dense in $\mathcal H$. This follows, for example, from the identity $\overline{\sum_{i=1}^{m} M_i^\perp} = M^\perp$; see \cite[Theorem 4.6]{Deutsch2001}. Moreover, it was suggested in \cite[Remark 4.4(b)]{BadeaSeifert2016} that when $\sum_{i=1}^{m} M_i^\perp$ is not closed, then, in general, the inclusion $X_1 \subset Y$ is strict.

\bigskip
After this short literature overview, we may now present the contributions of our paper, which are  as follows:

\begin{enumerate}[{\bfseries (C1)} ]
  \item  We show that for all $y_0 \in \mathcal H$,  the average distance \eqref{int:proxAVG} exhibits a polynomial rate $o(k^{-1/2})$.
  \item  Moreover, we show that for all $ y_0 \in Y $, the error term \eqref{int:error}, the  average  distance \eqref{int:proxAVG} and the increment \eqref{int:proxAR}  satisfy polynomial upper bounds $\mathcal O(k^{-1/2})$, $\mathcal O(k^{-1})$ and $\mathcal O(k^{-3/2})$, respectively.
  \item Furthermore, we prove that if $\sum_{i=1}^{m}M_i^\perp$ is not closed, then  all of the above-mentioned rates, including $o(k^{-1})$  in Theorem \ref{thm:ARMi}, cannot be improved.
  \item  In addition, we verify that if $\sum_{i=1}^{m}M_i^\perp$ is not closed, then the inclusions $X_{p+1} \subset X_p \subset Y$ are indeed strict for $p=1,2,\ldots$.
  \item  Finally, we demonstrate that if $\sum_{i=1}^{m}M_i^\perp$ is not closed, then there is a dense subset of starting points $V \subset Y$ on which
      \begin{equation}\label{int:limsup}
        \limsup_{k \to \infty} k^{1/2+\varepsilon}\|y_k - P_M(y_0)\| = \infty
      \end{equation}
      for all $\varepsilon > 0$. An analogous property holds for the average distance \eqref{int:proxAVG} and for the increment \eqref{int:proxAR}.
\end{enumerate}
The first three statements  \textbf{(C1)--(C3)}  can be found in Theorem \ref{thm:main}. In particular, we fully solve \cite[Problem 3]{BorodinKopecka2020} so that the upper bound in \eqref{thm:polyRate:msets} is $\mathcal O(k^{-1/2})$. It is worth pointing out that our ``optimality argument''  for \textbf{(C3)}  significantly differs from the one used in \cite[Theorem 3]{BorodinKopecka2020} and is based on Lemma \ref{lem:thresholdsEps}. In particular, this approach  not only allowed  us to verify \textbf{(C4)}, but also led us to \textbf{(C5)}; see Corollary \ref{cor:Xp} and Theorem \ref{thm:main2}. Note here that when $M = \{0\}$, then \eqref{int:limsup} implies \eqref{int:lowerBound} on $V$ for infinitely many $k$'s; see Remark \ref{rem:lowerBounds}.

\bigskip
Our paper is organized as follows. In Section \ref{sec:Pierra} we recall basic properties of the product space formulation of Pierra, on which we  rely  heavily  throughout  the paper. In Section \ref{sec:convexSets} we develop basic inequalities that connect \eqref{int:proxAVG} and \eqref{int:proxAR}. Note that the results of Section \ref{sec:convexSets}, in particular Lemma \ref{lem:convSets}, hold for all closed and convex sets and not only for closed and linear subspaces. Section \ref{sec:subspaces} is where we present our main results. In the \hyperref[sec:Appendix]{Appendix} we present the  proofs  of Theorems \ref{thm:ARMi} and \ref{thm:superPoly}.


\section{Product Space Formulation of Pierra} \label{sec:Pierra}
We consider the product space
\begin{equation}\label{}
  \mathbf H := \underbrace{\mathcal H \times \ldots \times \mathcal H}_{m \text{ times}}
\end{equation}
equipped with the inner product $\langle \cdot, \cdot \rangle$ and the induced norm $\|\cdot\|$ given by
\begin{equation}\label{}
  \langle \mathbf x, \mathbf y\rangle := \frac 1 m \sum_{i=1}^m \langle x_i, y_i\rangle
  \quad \text{and} \quad \|\mathbf x\| := \sqrt{\frac 1 m \sum_{i=1}^{m}\|x_i\|^2},
\end{equation}
where $\mathbf x = (x_1,\ldots,x_m),\ \mathbf y = (y_1,\ldots,y_m) \in \mathbf H$. In order to distinguish subsets and operators defined in $\mathcal H$ from those defined in $\mathbf H$, we use bold font in the latter case.
Following Pierra \cite{Pierra1984}, let
\begin{equation}\label{}
  \mathbf C := M_1  \times \cdots \times  M_m \quad \text{and} \quad \mathbf D := \{\mathbf x = \underbrace{(x,\ldots,x)}_{m \text{ times}} \colon x \in \mathcal H\}.
\end{equation}
The subspace $\mathbf D$ is called the \emph{diagonal} of $\mathbf H$. In addition, we define
\begin{equation}\label{}
  \mathbf M := \underbrace{M  \times \cdots \times  M}_{m \text{ times}}
  \quad \text{and} \quad
  \mathbf M_i := \underbrace{ M_i  \times \cdots \times  M_i}_{m \text{ times}},
\end{equation}
$i = 1,\ldots,m.$
It is not difficult to see that $\mathbf H$ is a Hilbert space while all of the above-mentioned sets are closed and linear subspaces of $\mathbf H$. Moreover, one can verify that
\begin{equation}\label{eq:Cperp}
  \mathbf C^\perp = M_1^\perp  \times \cdots \times  M_m^\perp,
\end{equation}
where ``$\perp$'' stands for the orthogonal complement in both $\mathcal H$ and $\mathbf H$; see, for example, \cite[Theorem 4.6]{Deutsch2001}. Furthermore, we have the following theorem:
\begin{theorem}\label{thm:PDPC}
  Let $\mathbf x = (x_1,\ldots,x_m)$ and let $s = \frac 1 m\sum_{i=1}^{m}x_i$. Then,
  \begin{equation}\label{thm:PDPC:eq}
  P_{\mathbf C}(\mathbf x) = (P_{M_i}(x_i))_{i=1}^m, \quad
  P_{\mathbf D}(\mathbf x) = (s)_{i=1}^m
  \quad\text{and}\quad
  P_{\mathbf C \cap \mathbf D}(\mathbf x) = (P_{M}(s))_{i=1}^m.
\end{equation}
\end{theorem}

\begin{proof}
See, for example, \cite[Lemma 1.1]{Pierra1984} or \cite[Section 4.4.1]{Cegielski2012}.
\end{proof}
Analogously to the projection onto $\mathbf C$, one can obtain coordinate-wise formulas for the orthogonal projections  $P_{\mathbf C^\perp}$ and $P_{\mathbf M_j}$,  $j=1,\ldots,m$,  that is,
\begin{equation}\label{eq:PCPMperp}
  P_{\mathbf C^\perp}(\mathbf x) = (P_{M_i^\perp}(x_i))_{i=1}^m
  \quad \text{and} \quad
  P_{\mathbf M_j}(\mathbf x) = (P_{M_j}(x_i))_{i=1}^m.
\end{equation}
In particular, the coordinate-wise formulas apply to the product of orthogonal projections in $\mathbf H$ defined by
\begin{equation}\label{def:Tbold}
   \mathbf T(\mathbf x) := P_{\mathbf M_m} \ldots P_{\mathbf M_1}(\mathbf x) = (T(x_i))_{i=1}^m.
\end{equation}
Note, however, that unlike the projection $P_{\mathbf C}$, the projections $P_{\mathbf M_i}$ do commute with $P_{\mathbf D}$.
\begin{proposition} \label{thm:commutingProj}
  For each $i = 1,\ldots, m$, we have
  \begin{equation}\label{thm:commutingProj:eq1}
    P_{\mathbf M_i} P_{\mathbf D} = P_{\mathbf D}P_{\mathbf M_i}.
  \end{equation}
  In particular,
  \begin{equation}\label{thm:commutingProj:eq2}
    \mathbf T P_{\mathbf D} = P_{\mathbf D}\mathbf T.
  \end{equation}
\end{proposition}

\begin{proof}
  Let $\mathbf x = (x_1,\ldots,x_m)$. Then,  by \eqref{thm:PDPC:eq} and \eqref{eq:PCPMperp}, we have
  \begin{equation}\label{pr:PSF2:PDPM}
    P_{\mathbf D}P_{\mathbf M_i}(\mathbf x)
    = \textstyle \left(\frac 1 m \sum_{j=1}^m P_{M_i}(x_j)\right)_{t=1}^m
    = \textstyle \left(P_{M_i}\left(\frac 1 m \sum_{j=1}^m x_j \right)\right)_{t=1}^m
    = P_{\mathbf M_i}P_{\mathbf D}(\mathbf x).
  \end{equation}
  Equation \eqref{thm:commutingProj:eq2} follows from the definition of $\mathbf T$.
\end{proof}

We finish this section with a few simple equalities and inequalities,  which are used in the sequel.

\begin{lemma}\label{thm:normsInPS}
  For each $k = 1,2,\ldots$, we have
  \begin{equation}\label{thm:normsInPS:ARPD}
    \|(\mathbf T^k - \mathbf T^{k-1})P_{\mathbf D}\| = \|T^k - T^{k-1}\| ,
  \end{equation}
  \begin{equation}\label{thm:normsInPS:PMiT}
    \|P_{\mathbf C^\perp}\mathbf T^k P_{\mathbf D} \|
    \leq \max_{i=1,\ldots,m} \|P_{M_i^\perp} T^k\|
    \leq \sqrt m \|P_{\mathbf C^\perp}\mathbf T^k P_{\mathbf D} \|,
  \end{equation}
  \begin{equation}\label{thm:normsInPS:TPMi}
    \|\mathbf T^k P_{\mathbf D}P_{\mathbf C^\perp} \|
    \leq \max_{i=1,\ldots,m} \|T^k P_{M_i^\perp}\|
    \leq \sqrt m \|\mathbf T^k P_{\mathbf D}P_{\mathbf C^\perp} \|,
  \end{equation}
  \begin{equation}\label{thm:normsInPS:ARPDPC}
    \|(\mathbf T^k -  \mathbf T^{k-1} ) P_{\mathbf D}P_{\mathbf C^\perp}\|
    \leq \max_{i=1,\ldots,m}\|(T^k - T^{k-1}) P_{M_i^\perp}\|
    \leq \sqrt m \|(\mathbf T^k -  \mathbf T^{k-1} ) P_{\mathbf D}P_{\mathbf C^\perp}\| ,
  \end{equation}
  \begin{equation}\label{thm:normsInPS:PMiTPMj}
    \|P_{\mathbf C^\perp}\mathbf T^k P_{\mathbf D}P_{\mathbf C^\perp}\|
    \leq \max_{i,j=1,\ldots,m}\|P_{M_j^\perp}T^k P_{M_i^\perp}\|
    \leq m \|P_{\mathbf C^\perp}\mathbf T^k P_{\mathbf D}P_{\mathbf C^\perp}\|.
  \end{equation}
\end{lemma}

\begin{proof}
  Equality \eqref{thm:normsInPS:ARPD} can be easily obtained by a direct calculation of the corresponding norms. Indeed, we have
  \begin{align}\label{} \nonumber
    \|(\mathbf T^k - \mathbf T^{k-1})P_{\mathbf D}\|
    & = \sup\{\|(\mathbf T^k - \mathbf T^{k-1})P_{\mathbf D}(\mathbf x)\| \colon \mathbf x = (x,\ldots,x) \in \mathbf D,\ \|\mathbf x\| \leq 1\}\\
    & = \sup\{\|(T^k - T^{k-1})(x)\| \colon x \in \mathcal H,\ \| x\| \leq 1\}.
  \end{align}

  In order to show \eqref{thm:normsInPS:PMiT}, let $\mathbf x = (x_1,\ldots,x_m) \in \mathbf H$. Then, by using the convexity of $\|\cdot\|^2$, we have
  \begin{equation}\label{pr:normsInPS:PMiT:1}
    \|P_{\mathbf C^\perp} \mathbf T^k P_{\mathbf D}(\mathbf x)\|^2
    = \frac 1 m \sum_{i=1}^{m} \left\|P_{M_i^\perp}T^k\left(\frac 1 m \sum_{j=1}^{m} x_j \right) \right\|^2
    \leq \max_{i=1,\ldots,m}\|P_{M_i^\perp}T^k\|^2 \cdot \|\mathbf x\|^2.
  \end{equation}
  On the other hand, for $x\in \mathcal H$ and $\mathbf x:=(x,\ldots,x)$ (so that $\|\mathbf x\| = \|x\|$), we have
  \begin{equation}\label{pr:normsInPS:PMiT:2}
    \|P_{M_i^\perp}T^k(x)\|^2
    \leq \sum_{i=1}^{m} \|P_{M_i^\perp}T^k(x)\|^2 = m \|P_{\mathbf C^\perp} \mathbf T^k P_{\mathbf D}(\mathbf x)\|^2
    \leq m \|P_{\mathbf C^\perp} \mathbf T^k P_{\mathbf D}\|^2 \cdot \|x\|^2.
  \end{equation}
  It now suffices to take the supremum over $\|\mathbf x\| = 1$ in \eqref{pr:normsInPS:PMiT:1} and over $\|x\| = 1$ in \eqref{pr:normsInPS:PMiT:2}.

  Inequalities \eqref{thm:normsInPS:TPMi} follow from \eqref{thm:normsInPS:PMiT}. Indeed, if we change the order of projections in \eqref{thm:normsInPS:PMiT}, for example, by using a permutation $\sigma = (\sigma(1), \ldots, \sigma(m))$, then, the corresponding operators $\mathbf T_\sigma := P_{\mathbf M_{\sigma(m)}} \ldots P_{\mathbf M_{\sigma(1)}}$ and $T_\sigma := P_{M_{\sigma(m)}} \ldots P_{M_{\sigma(1)}}$ satisfy
  \begin{equation}\label{pr:normsInPS:PMiTsigma}
    \|P_{\mathbf C^\perp}\mathbf T_\sigma^k P_{\mathbf D} \|
    \leq \max_{i=1,\ldots,m} \|P_{M_i^\perp} T_\sigma^k\|
    \leq \sqrt m \|P_{\mathbf C^\perp}\mathbf T_\sigma^k P_{\mathbf D} \|.
  \end{equation}
  In particular, for the adjoints $\mathbf T^* = P_{\mathbf M_1} \ldots P_{\mathbf M_m}$ and $T^* = P_{M_1} \ldots P_{M_m}$, we get
  \begin{equation}\label{pr:normsInPS:PMiTadjoint}
    \|P_{\mathbf C^\perp}(\mathbf T^*)^k P_{\mathbf D} \|
    \leq \max_{i=1,\ldots,m} \|P_{M_i^\perp} (T^*)^k\|
    \leq \sqrt m \|P_{\mathbf C^\perp}(\mathbf T^*)^k P_{\mathbf D} \|,
  \end{equation}
  Using the equality between the norms of a bounded linear operator and its adjoint, and by Proposition \ref{thm:commutingProj}, we get
  \begin{equation}\label{pr:normsInPS:TPMi}
    \|P_{\mathbf C^\perp}(\mathbf T^*)^k P_{\mathbf D} \|
    = \|(P_{\mathbf C^\perp}(\mathbf T^*)^k P_{\mathbf D} )^*\|
    = \|\mathbf T^k P_{\mathbf D}P_{\mathbf C^\perp} \|
  \end{equation}
  and
  \begin{equation}\label{}
    \|P_{M_i^\perp} (T^*)^k\|
    = \|(P_{M_i^\perp} (T^*)^k)^*\|
    = \|T^k P_{M_i^\perp}\|,
  \end{equation}
  which, when combined with \eqref{pr:normsInPS:PMiTadjoint} proves \eqref{thm:normsInPS:TPMi}.

  We now proceed to showing \eqref{thm:normsInPS:ARPDPC}. The proof is a combination of arguments used for \eqref{thm:normsInPS:PMiT} and \eqref{thm:normsInPS:TPMi} with $\mathbf T^k$ replaced by $\mathbf T^k - \mathbf T^{k-1}$ and with $T^k$ replaced by $T^k - T^{k-1}$. Indeed, observe that by repeating the calculation from \eqref{pr:normsInPS:PMiT:1} and \eqref{pr:normsInPS:PMiT:2}, we get
  \begin{equation}\label{pr:normsInPS:ARPDPC}
    \|P_{\mathbf C^\perp}(\mathbf T^k - \mathbf T^{k-1} ) P_{\mathbf D}\|
    \leq \max_{i=1,\ldots,m}\|P_{M_i^\perp}(T^k - T^{k-1}) \|
    \leq \sqrt m \|P_{\mathbf C^\perp}(\mathbf T^k -  \mathbf T^{k-1} ) P_{\mathbf D}\|.
  \end{equation}
  Obviously, inequalities \eqref{pr:normsInPS:ARPDPC} hold true if we change the order of projections by using the operators $\mathbf T_\sigma$ and $T_\sigma$; compare with \eqref{pr:normsInPS:PMiTsigma}. In particular, \eqref{pr:normsInPS:ARPDPC} holds for the adjoints $\mathbf T^*$ and $T^*$. Knowing that
  \begin{equation}\label{}
    \|P_{\mathbf C^\perp}( (\mathbf T^*)^k - (\mathbf T^*)^{k-1} ) P_{\mathbf D}\|
    = \|(\mathbf T^k - \mathbf T^{k-1} ) P_{\mathbf D}P_{\mathbf C^\perp}\|
  \end{equation}
  and
  \begin{equation}\label{}
    \|P_{M_i^\perp}((T^*)^k - (T^*)^{k-1}) \| = \|(T^k - T^{k-1}) P_{M_i^\perp}\|,
  \end{equation}
  we arrive at \eqref{thm:normsInPS:ARPDPC}, as claimed.

  Finally, we proceed to showing inequalities \eqref{thm:normsInPS:PMiTPMj}. On the one hand, using the convexity of $\|\cdot\|^2$, for $\mathbf x = (x_1,\ldots,x_m) \in \mathbf H$, we get
  \begin{align}\label{pr:normsInPS:PMiTPMj:1} \nonumber
    \|P_{\mathbf C^\perp} \mathbf T^k P_{\mathbf D}P_{\mathbf C^\perp}(\mathbf x)\|^2
    & = \frac 1 m \sum_{i=1}^{m} \left\|P_{M_i^\perp}T^k \left(\frac 1 m \sum_{j=1}^{m} P_{M_j^\perp}(x_j)\right) \right\|^2\\ \nonumber
    & \leq \frac 1 {m^2} \sum_{i=1}^{m} \sum_{j=1}^{m} \|P_{M_i^\perp}T^k P_{M_j^\perp}(x_j)\|^2 \\
    & \leq \max_{i,j=1,\ldots,m}\|P_{M_j^\perp}T^k P_{M_i^\perp}\|^2 \cdot \|\mathbf x\|^2.
  \end{align}
  On the other hand, for each $x \in \mathcal H$ and for $\mathbf x_j := (0,\ldots,mx,\ldots,0) \in \mathbf H$, we have
  \begin{equation}\label{}
    \mathbf P_{\mathbf C^\perp} \mathbf T^k P_{\mathbf D}P_{\mathbf C^\perp}(\mathbf x_j)
    = ( P_{M_1^\perp}T^kP_{M_j^\perp}(x), \ldots,P_{M_m^\perp} T^kP_{M_j^\perp}(x))
  \end{equation}
  and
  \begin{align}\label{pr:normsInPS:PMiTPMj:2} \nonumber
    \|P_{M_i^\perp}T^kP_{M_j^\perp}(x)\|^2
    & \leq \sum_{i=1}^{m} \|P_{M_i^\perp}T^kP_{M_j^\perp}(x)\|^2 \\ \nonumber
    & = m \|\mathbf P_{\mathbf C^\perp} \mathbf T^k P_{\mathbf D}P_{\mathbf C^\perp}(\mathbf x_j) \|^2 \\
    & \leq m^2 \|\mathbf P_{\mathbf C^\perp} \mathbf T^k P_{\mathbf D}P_{\mathbf C^\perp}\|^2 \cdot \|x\|^2,
  \end{align}
  as $\|\mathbf x_j\| = \sqrt m\|x\|$. Thus, by taking the supremum over $\|\mathbf x\| = 1$ in \eqref{pr:normsInPS:PMiTPMj:1} and over $\|x\| = 1$ in \eqref{pr:normsInPS:PMiTPMj:2}, we arrive at \eqref{thm:normsInPS:PMiTPMj}.
\end{proof}

\section{Closed and Convex Subsets} \label{sec:convexSets}
Throughout this section we assume that for each $i = 1, \ldots, m$, the set $C_i$ is a closed and convex subset of $\mathcal H$, and we put $C := \bigcap_{i=1}^m C_i$. The following lemma corresponds to \cite[Lemma 8 (iii)]{BargetzReichZalas2018}.

\begin{lemma}\label{lem:convSets}
  Let the sequence $\{y_k\}_{k=0}^\infty$ be defined by the method of cyclic projections using the subsets $C_i$, that is,
  \begin{equation}\label{lem:convSets:yk}
    y_0 \in \mathcal H, \quad y_{k} := (P_{C_m} \ldots P_{C_1})^k (y_0), \quad k = 1,2,\ldots.
  \end{equation}
  Assume that the intersection $C \neq \emptyset$. Then, for each $k=1,2,\ldots,$ we have
  \begin{equation}\label{lem:convSets:ineq1}
    \frac 1 m \sum_{i=1}^{m} d(y_k, C_i)^2
    \leq \frac m 2 \|y_k - y_{k-1}\| \cdot d(y_{k-1}, C)
  \end{equation}
  while
  \begin{equation}\label{lem:convSets:ineq2}
    \max_{i=1,\ldots,m} d(y_k, C_i)^2 \leq m \|y_k - y_{k-1}\| \cdot d(y_{k-1}, C).
  \end{equation}
\end{lemma}

\begin{proof}
  We follow the argument from \cite[Lemma 8 (iii)]{BargetzReichZalas2018} which we adjust for a simple product of the nearest point projections.
  Put $Q_0 := I$ (the identity operator) and $Q_i := P_{C_i} \ldots P_{C_1}$, $i = 1, \ldots, m$. By using the properties of the projections $P_{C_i}$ (see \cite[Corollaries 2.2.24 and 4.5.2]{Cegielski2012}),  for each $z \in C$, we have
  \begin{equation}\label{pr:convSets:Qi1}
    \sum_{i=1}^{m} \|Q_i(y_k) - Q_{i-1}(y_k)\|^2 \leq \|y_k - z\|^2 - \|y_{k+1} - z\|^2.
  \end{equation}
  Moreover, by the Cauchy-Schwarz inequality, we get
  \begin{align}\label{pr:convSets:Qi2} \nonumber
    \|y_{k+1} - z\|^2 & = \|y_{k+1} - y_k\|^2 + \|y_k - z\|^2 + 2 \langle y_{k+1} - y_k, y_k - z\rangle \\
    &\geq \|y_{k+1} - y_k\|^2 + \|y_k - z\|^2 - 2 \|y_{k+1} - y_k\| \cdot \|y_k - z\|.
  \end{align}
  In particular, by combining \eqref{pr:convSets:Qi1} and \eqref{pr:convSets:Qi2}, we obtain
  \begin{equation}\label{pr:convSets:Qi3}
    \sum_{i=1}^{m} \|Q_i(y_k) - Q_{i-1}(y_k)\|^2 \leq 2 \|y_{k+1} - y_k\| \cdot \|y_k - z\|.
  \end{equation}
  Furthermore, since the product of projections $Q_m$ is nonexpansive and $\fix Q_m = C$, for all $k=1,2,\ldots$, we have
  \begin{equation}\label{pr:convSets:Fejer}
    \|y_{k+1} - y_k\| \leq \|y_k - y_{k-1}\| \quad \text{and} \quad \|y_k - z\| \leq \|y_{k-1} - z\|.
  \end{equation}

  Let $j \in \{1,\ldots,\lfloor \frac m 2 \rfloor \}$. Then, by the definition of the metric projection, by using the triangle and the Cauchy-Schwarz inequalities, and by combining this with \eqref{pr:convSets:Qi3}--\eqref{pr:convSets:Fejer}, we obtain
  \begin{align}\label{pr:convSets:dCj1} \nonumber
    d(y_k, C_j)^2 & = \|y_k - P_{C_j}(y_k)\|^2 \leq \|y_k - Q_j(y_k)\|^2 \\ \nonumber
      & \leq \left( \sum_{i=1}^{j} \|Q_i(y_k) - Q_{i-1}(y_k)\| \right)^2 \\ \nonumber
      & \leq j \sum_{i=1}^{j} \|Q_i(y_k) - Q_{i-1}(y_k)\|^2\\
      & \leq 2j \cdot \|y_{k} - y_{k-1}\| \cdot \|y_{k-1} - z\|.
  \end{align}
  Let now $j \in \{\lfloor \frac m 2 \rfloor + 1,\ldots, m-1\}$. By using similar arguments, we obtain
  \begin{align}\label{pr:convSets:dCj2} \nonumber
    d(y_k, C_j)^2 & = \|Q_m (y_{k-1}) - P_{C_j}(Q_m (y_{k-1}))\|^2 \leq \|Q_m (y_{k-1}) - Q_j(y_{k-1})\|^2 \\ \nonumber
      & \leq \left( \sum_{i=j+1}^{m} \|Q_i(y_{k-1}) - Q_{i-1}(y_{k-1})\| \right)^2 \\ \nonumber
      & \leq (m-j) \sum_{i=j+1}^{m} \|Q_i(y_{k-1}) - Q_{i-1}(y_{k-1})\|^2\\
      & \leq 2(m-j) \cdot \|y_{k} - y_{k-1}\| \cdot \|y_{k-1} - z\|.
  \end{align}
  By combining \eqref{pr:convSets:dCj1} and \eqref{pr:convSets:dCj2}, we arrive at
  \begin{equation}\label{}
    \frac 1 m \sum_{i=1}^{m} d(y_k, C_i)^2
    \leq s_m \|y_{k} - y_{k-1}\| \cdot \|y_{k-1} - z\|,
  \end{equation}
  where
  \begin{equation}\label{}
    s_m = \frac 2 m \left( \sum_{i=1}^{\lfloor \frac m 2 \rfloor} i + \sum^{m-1}_{i=  \lfloor \frac m 2 \rfloor + 1}(m-i)\right)
    =
    \begin{cases}
      m/2, & \text{if $m$ is even}\\
      m/2-1/(2m), & \text{if $m$ is odd}.
    \end{cases}
  \end{equation}
  This shows inequality \eqref{lem:convSets:ineq1}. Inequality \eqref{lem:convSets:ineq2} follows directly from \eqref{pr:convSets:dCj1} and \eqref{pr:convSets:dCj2}.
\end{proof}

\begin{theorem} \label{thm:CPMCi}
  Let the sequence $\{y_k\}_{k=0}^\infty$ be defined as in Lemma \ref{lem:convSets} and assume that $C \neq \emptyset$. Then, we have
  \begin{equation}\label{thm:CPMCi:rateAR}
    \|y_{k} - y_{k-1}\| = o(k^{-1/2})
  \end{equation}
  and
  \begin{equation}\label{thm:CPMCi:rateMaxAVD}
    {\textstyle\sqrt{\frac 1 m \sum_{i=1}^m d^2(y_k, C_i)}} =  o(k^{-1/4}).
  \end{equation}
\end{theorem}

\begin{proof}
  Let $z \in C$. Knowing that $P_{C_m} \ldots P_{C_1}$ is $(1/m)$-strongly quasi-nonexpansive, we have
  \begin{equation}\label{}
    \|y_{k+1} - y_k\|^2 \leq m(\|y_k - z\|^2 - \|y_{k+1} - z\|^2).
  \end{equation}
  see, for example,  \cite[Corollary 4.5.3]{Cegielski2012}.  Consequently,
  \begin{equation}\label{}
    \sum_{k=1}^{\infty} \|y_{k} - y_{k-1}\|^2 \leq m \|y_0 - z\|^2  < \infty.
  \end{equation}
   In particular, $\sum_{n=k}^\infty \|y_n - y_{n-1}\|^2 \to 0$ as $k \to \infty$.  By \eqref{pr:convSets:Fejer}, we have
  \begin{equation}\label{}
    \frac k 2  \|y_{k} - y_{k-1}\|^2
     \leq \left\lceil \frac k 2 \right\rceil \|y_{k} - y_{k-1}\|^2
    \leq \sum_{ {n= \lfloor k/2\rfloor+1} }^{k} \|y_{n} - y_{n-1}\|^2 \to 0
  \end{equation}
  as $k \to \infty$. This proves \eqref{thm:CPMCi:rateAR}. The rate of \eqref{thm:CPMCi:rateMaxAVD} follows immediately from Lemma \ref{lem:convSets}.
\end{proof}

\section{Closed and Linear Subspaces} \label{sec:subspaces}
In this section we oftentimes use the  product space  notation introduced in Section \ref{sec:Pierra}. The following result is a direct consequence of  Theorem \ref{thm:ARMi}, Lemma \ref{thm:normsInPS} and Lemma \ref{lem:convSets}.

\begin{lemma} \label{lem:thresholds}
  For the operators $T$ defined in \eqref{int:T} and $\mathbf T$ defined in \eqref{def:Tbold}, we have:
  \begin{multicols}{2}
  \begin{enumerate}[(i)]
    \item $\|T^k - T^{k-1}\| = \mathcal O(k^{-1})$;
    \item $\max\limits_{i=1,\ldots,m} \|P_{M_i^\perp}T^k\| = \mathcal O(k^{-1/2})$;
    \item $\max\limits_{i=1,\ldots,m} \|T^k P_{M_i^\perp}\| = \mathcal O(k^{-1/2})$;
    \item $\max\limits_{i=1,\ldots,m} \|(T^k-T^{k-1}) P_{M_i^\perp}\| = \mathcal O(k^{-3/2})$;
    \item $\max\limits_{i,j=1,\ldots,m} \|P_{M_i^\perp} T^k P_{M_j^\perp}\| = \mathcal O(k^{-1})$;

    \item $\|(\mathbf T^k - \mathbf T^{k-1})P_{\mathbf D}\| = \mathcal O(k^{-1})$;
    \item $\|P_{\mathbf C^\perp} \mathbf T^k P_{\mathbf D}\| = \mathcal O(k^{-1/2})$;
    \item $\|\mathbf T^k P_{\mathbf D} P_{\mathbf C^\perp} \| = \mathcal O(k^{-1/2})$;
    \item $\|(\mathbf T^{k} - \mathbf T^{k-1})P_{\mathbf D}P_{\mathbf C^\perp}\| = \mathcal O(k^{-3/2})$;
    \item $\|P_{\mathbf C^\perp}\mathbf T^k P_{\mathbf D} P_{\mathbf C^\perp}\| = \mathcal O(k^{-1})$.
  \end{enumerate}
  \end{multicols}
\end{lemma}

\begin{proof}
   Statement \emph{(i)} follows from Theorem \ref{thm:ARMi} and the uniform boundedness principle \cite[Theorem 2.2]{Brezis2011}.   In view of Lemma \ref{thm:normsInPS}, it suffices to show statements \emph{(vii)}--\emph{(x)}.

  \bigskip
  \emph{(vii).} We show that
  \begin{equation}\label{pr:thresholds:vii:ToShow}
    \|P_{\mathbf C^\perp} \mathbf T^k P_{\mathbf D}\|
    \leq \sqrt{\frac m 2 \|(\mathbf T^k - \mathbf T^{k-1})P_{\mathbf D}\|}.
  \end{equation}
  Let $\mathbf x = (x_1,\ldots,x_m) \in \mathbf H$. Moreover, let $\{y_k\}_{k=0}^\infty$ be defined by the method of cyclic projections \eqref{int:yk} with $y_0 := \frac 1 m \sum_{i=1}^{m}x_i$. Then, using the coordinate-wise projection formulas from Section \ref{sec:Pierra} and Lemma \ref{lem:convSets} (see \eqref{lem:convSets:ineq1}), we obtain
  \begin{align}\label{pr:thresholds:vii:avgDist} \nonumber
    \|P_{\mathbf C^\perp} \mathbf T^k P_{\mathbf D}(\mathbf x)\|^2
    & = \frac 1 m \sum_{i=1}^{m} \|P_{M_i^\perp}(y_k)\|^2
    = \frac 1 m \sum_{i=1}^{m} d^2(y_k, M_i)\\ \nonumber
    & \leq \frac m 2 \|y_k - y_{k-1}\| \cdot \|y_{k-1}\| \\ \nonumber
    & = \frac m 2 \|(\mathbf T^k - \mathbf T^{k-1})P_{\mathbf D}(\mathbf x)\| \cdot \|\mathbf T^{k-1}P_{\mathbf D}(\mathbf x)\| \\
    & \leq \frac m 2 \|(\mathbf T^k - \mathbf T^{k-1})P_{\mathbf D}\| \cdot \|\mathbf x\|^2.
  \end{align}

  \bigskip
  \emph{(viii).} We show that
  \begin{equation}\label{pr:thresholds:viii:ToShow}
    \|\mathbf T^k P_{\mathbf D} P_{\mathbf C^\perp} \| \leq \sqrt{\frac m 2 \|(\mathbf T^k - \mathbf T^{k-1})P_{\mathbf D}\|},
  \end{equation}
  where we use an argument similar to the one used in the proof of \eqref{thm:normsInPS:TPMi}. Indeed, observe that if we change the order of projections in \eqref{pr:thresholds:vii:ToShow} by using a permutation $\sigma = (\sigma(1), \ldots, \sigma(m))$, then the operator $\mathbf T_\sigma := P_{\mathbf M_{\sigma(m)}} \ldots P_{\mathbf M_{\sigma(1)}}$ satisfies
  \begin{equation}
    \|P_{\mathbf C^\perp} \mathbf T_\sigma^k P_{\mathbf D}\|
    \leq \sqrt{\frac m 2 \|(\mathbf T_\sigma^k - \mathbf T_\sigma^{k-1})P_{\mathbf D}\|}.
  \end{equation}
  In particular, for the adjoint $\mathbf T^*$, we get
  \begin{equation}\label{pr:thresholds:viii:1}
    \|P_{\mathbf C^\perp}  (\mathbf T^*)^k P_{\mathbf D}\|
    \leq \sqrt{\frac m 2 \|(\mathbf T^*)^k-(\mathbf T^*)^{k-1}) P_{\mathbf D}\|}.
  \end{equation}
  Note that in view of Proposition \ref{thm:commutingProj}, the projection $P_{\mathbf D}$ commutes with the operator $\mathbf T$. Moreover, using the equality between the norms of a bounded linear operator and its adjoint, we get
  \begin{equation}\label{pr:thresholds:viii:2}
    \|\mathbf T^k P_{\mathbf D} P_{\mathbf C^\perp} \|
    = \|(\mathbf T^k P_{\mathbf D} P_{\mathbf C^\perp} )^*\|
    = \|P_{\mathbf C^\perp}  (\mathbf T^*)^k P_{\mathbf D}\|
  \end{equation}
  and
  \begin{equation}\label{}
    \|(\mathbf T^*)^k-(\mathbf T^*)^{k-1}) P_{\mathbf D}\|
    = \|(\mathbf T^k - \mathbf T^{k-1})P_{\mathbf D}\|.
  \end{equation}
  This proves \emph{(viii)}.

  \bigskip
  \emph{(ix).} Since $P_{\mathbf D}$ is idempotent and commutes with $\mathbf T$, we have
  \begin{align}\label{pr:thresholds:ix} \nonumber
    \|(\mathbf T^k - \mathbf T^{k-1}) P_{\mathbf D} P_{\mathbf C^\perp}\|
    & = \|(( \mathbf T^{\lfloor k/2 \rfloor} - \mathbf T^{\lfloor k/2 \rfloor-1})P_{\mathbf D})
    \  (\mathbf T^{\lceil k/2 \rceil-1} P_{\mathbf D} P_{\mathbf C^\perp})\| \\
    & \leq \|( \mathbf T^{\lfloor k/2 \rfloor} - \mathbf T^{\lfloor k/2 \rfloor-1})P_{\mathbf D}\|
    \cdot  \|(\mathbf T^{\lceil k/2 \rceil-1} P_{\mathbf D} P_{\mathbf C^\perp})\|.
  \end{align}
  By combining this with \emph{(vi)} and \emph{(viii)} we arrive at \emph{(ix)}.

  \bigskip
  \emph{(x).} We show that
  \begin{equation}\label{pr:thresholds:x:ToShow}
    \|P_{\mathbf C^\perp} \mathbf T^k P_{\mathbf D} P_{\mathbf C^\perp}\|
    \leq \sqrt{\frac m 2 \|(\mathbf T^k - \mathbf T^{k-1})P_{\mathbf D}P_{\mathbf C^\perp}\| \cdot \|\mathbf T^kP_{\mathbf D}P_{\mathbf C^\perp}\| }.
  \end{equation}
  We slightly adjust the argument from the proof of case \emph{(vii)}. Indeed , let $\mathbf x = (x_1,\ldots,x_m) \in \mathbf H$. Moreover, let $\{y_k\}_{k=0}^\infty$ be defined by the method of cyclic projections \eqref{int:yk}, but this time with $y_0 := \frac 1 m \sum_{i=1}^{m} P_{M_i^\perp}x_i$. Then, we obtain
  \begin{align}\label{pr:thresholds:vii:avgDist} \nonumber
    \|P_{\mathbf C^\perp} \mathbf T^k P_{\mathbf D}P_{\mathbf C^\perp}(\mathbf x)\|^2
    & = \frac 1 m \sum_{i=1}^{m} \|P_{M_i^\perp}(y_k)\|^2
    = \frac 1 m \sum_{i=1}^{m} d^2(y_k, M_i)\\ \nonumber
    & \leq \frac m 2 \|y_k - y_{k-1}\| \cdot \|y_{k-1}\| \\ \nonumber
    & = \frac m 2 \|(\mathbf T^k - \mathbf T^{k-1})P_{\mathbf D}P_{\mathbf C^\perp}(\mathbf x)\| \cdot \|\mathbf T^{k-1}P_{\mathbf D}P_{\mathbf C^\perp}(\mathbf x)\| \\
    & \leq \frac m 2 \|(\mathbf T^k - \mathbf T^{k-1})P_{\mathbf D}P_{\mathbf C^\perp}\| \cdot \|\mathbf T^{k-1}P_{\mathbf D}P_{\mathbf C^\perp}\| \cdot \|\mathbf x\|^2,
  \end{align}
  which shows \eqref{pr:thresholds:x:ToShow}. This, in view of \emph{(viii)} and \emph{(ix)}, completes the proof.
\end{proof}

In our next result we show that the thresholds established in Lemma \ref{lem:thresholds} are critical as they distinguish polynomial from linear rates of convergence.

\begin{lemma} \label{lem:thresholdsEps}
  Let $\varepsilon > 0$ and assume that  for the operators $T$ defined in \eqref{int:T} and $\mathbf T$ defined in \eqref{def:Tbold}, one of the following conditions holds:
  \begin{multicols}{2}
  \begin{enumerate}[(i)]
    \item $\|T^k - T^{k-1}\| = \mathcal O(k^{-1-\varepsilon})$;
    \item $\max\limits_{i=1,\ldots,m} \|P_{M_i^\perp}T^k\| = \mathcal O(k^{-1/2-\varepsilon})$;
    \item $\max\limits_{i=1,\ldots,m} \|T^k P_{M_i^\perp}\| = \mathcal O(k^{-1/2-\varepsilon})$;
    \item $\nolinebreak{\max\limits_{i=1,\ldots,m} \|(T^k-T^{k-1}) P_{M_i^\perp}\| = \mathcal O(k^{-3/2-\varepsilon})}$;
    \item $\max\limits_{i,j=1,\ldots,m} \|P_{M_i^\perp} T^k P_{M_j^\perp}\| = \mathcal O(k^{-1-\varepsilon})$;

    \item $\|(\mathbf T^k - \mathbf T^{k-1})P_{\mathbf D}\| = \mathcal O(k^{-1-\varepsilon})$;
    \item $\|P_{\mathbf C^\perp} \mathbf T^k P_{\mathbf D}\| = \mathcal O(k^{-1/2-\varepsilon})$;
    \item $\|\mathbf T^k P_{\mathbf D} P_{\mathbf C^\perp} \| = \mathcal O(k^{-1/2-\varepsilon})$;
    \item ${\|(\mathbf T^{k} - \mathbf T^{k-1})P_{\mathbf D}P_{\mathbf C^\perp}\| = \mathcal O(k^{-3/2-\varepsilon})}$;
    \item $\|P_{\mathbf C^\perp}\mathbf T^k P_{\mathbf D} P_{\mathbf C^\perp}\| = \mathcal O(k^{-1-\varepsilon})$.
  \end{enumerate}
  \end{multicols}
Then $\sum_{i=1}^{m} M_i^\perp$ is closed (equivalently, $\|T - P_M\| < 1$). In particular, all of the above-mentioned rates are linear and take the form $\mathcal O(q^k)$ for some $q \in (0,1)$.
\end{lemma}

\begin{proof}
   The road map of the proof is to show the following implications:
  \begin{equation}\label{}
    (i) \Rightarrow
    \textstyle \sum_{i=1}^{m}M_i^\perp \text{ is closed}, \quad
    \{(ii), (iii), (v)\} \Rightarrow (i) \quad \text{and} \quad
    (iv) \Rightarrow (v).
  \end{equation}
  The equivalences $(i)\Leftrightarrow (vi)$, $(ii)\Leftrightarrow (vii)$, $(iii)\Leftrightarrow (viii)$, $(iv)\Leftrightarrow (ix)$ and $(v)\Leftrightarrow (x)$ follow from Lemma \ref{thm:normsInPS}.

  \bigskip
  ``$(i) \Rightarrow \textstyle \sum_{i=1}^{m}M_i^\perp$ \emph{is closed.}'' Assume that $\|T^k - T^{k-1}\| = \mathcal O(k^{-1-\varepsilon})$ for some $\varepsilon >0$. Then there are $N \geq 1$ and $C >0$, such that
  \begin{equation}\label{}
    q := \sum_{n=N}^{\infty} \|T^n-T^{n-1}\| \leq C \sum_{n=N}^{\infty}\frac{1}{n^{1+\varepsilon}} < 1.
  \end{equation}
  In particular, by using the triangle inequality, for each $x \in \mathcal H$, $\|x\| = 1$, and for all $k = 1,2,\ldots,$ we have
  \begin{equation}\label{pr:betaEquiv:ineq}
    \|T^N(x) - T^{N+k}(x)\| \leq \|T^N - T^{N+k}\|
    \leq \sum_{n=N+1}^{N+k}\|T^n - T^{n-1}\| \leq q.
  \end{equation}
  Using Theorem \ref{thm:norm}, we see that $\lim_{k\to \infty} T^{N+k}(x) = P_{M}(T^N(x))$. On the other hand, recall that $P_M P_{M_i} = P_M$ for all $i=1,\ldots,m$ (see \cite[Lemma 9.2]{Deutsch2001}). Thus $P_M(T^N(x)) = P_M(x)$. Therefore, by passing to the limit as $k \to \infty$ and then, by taking the supremum over $\|x\| = 1$ on the left-hand side of \eqref{pr:betaEquiv:ineq}, we arrive at
  \begin{equation}\label{}
    \|T^N - P_M\| \leq q < 1.
  \end{equation}
  By applying Theorem \ref{thm:linear} to $T^N$ (seen as the product of $m \cdot N$ projections) and $M$ (seen as the intersection of $m\cdot N$ subspaces), we get
  \begin{equation}\label{}
    \sum_{i=1}^{m}M_i^\perp = \underbrace{\sum_{i=1}^{m}M_i^\perp  + \cdots +  \sum_{i=1}^{m}M_i^\perp}_{N \text{ times}} \text{ is closed},
  \end{equation}
  which completes the proof of the implication.

  \bigskip
  ``$\{(ii), (iii), (v)\} \Rightarrow (i)$''. We begin by showing that
  \begin{equation}\label{pr:thresholdsEps:ineq1}
    \|T^k-T^{k-1}\|^2 \leq \frac C n \cdot \max_{i,j=1,\ldots,m} \|P_{M_i^\perp} T^{2n} P_{M_j^\perp}\|
  \end{equation}
  for some $C >0$, where $k \geq 4$ and where $n := \lfloor (k-1)/3\rfloor$ (so that $3n \leq k-1$). Indeed, let $x \in \mathcal H$ be such that $\|x\| = 1$. Observe that $I - T$ commutes with $T$ and that
  \begin{equation}\label{pr:thresholdsEps:I-T}
     I-T = \sum_{i=1}^{m}(Q_{i-1} - Q_i) = \sum_{i=1}^{m}P_{M_i^\perp}Q_{i-1},
  \end{equation}
  where $Q_0 := I$ and $Q_i := P_{M_i}\ldots P_{M_1}$, $i = 1,\ldots,m$. Using the fact that the orthogonal projection is idempotent and self-adjoint, and that $\|T\| \leq 1$, we get
  \begin{align}\label{} \nonumber
    \|T^k(x) -  T^{k-1}(x)\|^2  & \leq \|(I-T)T^{3n}(x)\|^2  = \langle (I-T)T^{3n}(x), T^{2n}(I-T)T^n(x)\rangle \\ \nonumber
    & = \left\langle \sum_{i=1}^{m}P_{M_i^\perp}Q_{i-1} T^{3n}(x),\ T^{2n} \left(\sum_{j=1}^{m} P_{M_j^\perp}Q_{j-1} \right)
    T^n(x)\right \rangle \\ \nonumber
    & = \sum_{i=1}^{m} \sum_{j=1}^{m} \left\langle P_{M_i^\perp}Q_{i-1} T^{3n}(x),\ (P_{M_i^\perp} T^{2n} P_{M_j^\perp}) P_{M_j^\perp}Q_{j-1} T^n(x)\right \rangle \\
    & \leq m^2 \max_{i=1,\ldots,m} \|P_{M_i^\perp}Q_{i-1} T^{n}\|^2 \cdot \max_{i,j=1,\ldots,m} \|P_{M_i^\perp} T^{2n} P_{M_j^\perp}\|.
  \end{align}
  On the other hand, by Lemma \ref{lem:thresholds} (ii) applied to different orders of projections, we have
  \begin{equation}\label{pr:thresholdsEps:maxDist}
     \max_{i=1,\ldots,m} \|P_{M_i}^\perp Q_{i-1}T^{n}\|^2
     \leq \max_{i=1,\ldots,m} \|P_{M_i}^\perp (Q_{i-1}P_{M_m}\ldots P_{M_i})^{n}\|^2
     \leq \frac {C'} n,
  \end{equation}
  for some $C '>0$, which shows \eqref{pr:thresholdsEps:ineq1}.

  \bigskip
  Assume now that condition $(v)$ holds, that is,
  \begin{equation}\label{}
    \max_{i,j=1,\ldots,m} \|P_{M_i^\perp} T^{k} P_{M_j^\perp}\| \leq \frac{C'}{k^{1+\varepsilon}}
  \end{equation}
  for some $\varepsilon>0$ and some $C'>0$. Then,  using  \eqref{pr:thresholdsEps:ineq1} and knowing that $n \geq k/4$ (so that $1/n \leq 4/k$), we get
  \begin{equation}\label{}
    \|T^k - T^{k-1}\|^2 \leq \frac{C}{n} \cdot \frac{C'}{(2n)^{1+\varepsilon}}
    \leq \frac{4^{2+\varepsilon}C C'}{k^{2+\varepsilon}}.
  \end{equation}
  Thus we have arrived at condition $(i)$.

  \bigskip
  Assume now that condition $(ii)$ holds, that is,
  \begin{equation}\label{}
    \max_{i=1,\ldots,m} \|P_{M_i^\perp} T^{k} \| \leq \frac{C'}{k^{1/2+\varepsilon}}
  \end{equation}
  for some $\varepsilon>0$ and some $C'>0$. Then,
  \begin{align}\label{pr:thresholdsEps:ineq2}\nonumber
    \max_{i,j=1,\ldots,m} \|P_{M_i^\perp} T^{2n} P_{M_j^\perp}\|
    & \leq \max_{i=1,\ldots,m} \|P_{M_i^\perp} T^{n}\| \cdot \max_{j=1,\ldots,m} \|T^{n} P_{M_j^\perp}\| \\
    & \leq \frac{C'}{n^{1/2+\varepsilon}} \max_{j=1,\ldots,m} \|T^{n} P_{M_j^\perp}\|.
  \end{align}
  By Lemma \ref{lem:thresholds} (iii) we know that
  \begin{equation}\label{}
    \max_{j=1,\ldots,m} \|T^{n} P_{M_j^\perp}\| \leq \frac{C''}{\sqrt n}
  \end{equation}
  for some $C'' >0$. This, when combined with \eqref{pr:thresholdsEps:ineq1}, and with the inequality $n \geq k/4$, leads to
  \begin{equation}\label{}
    \|T^k-T^{k-1}\|^2 \leq \frac C n \cdot \frac{C'}{n^{1/2+\varepsilon}} \cdot \frac{C''}{n^{1/2}}
    \leq \frac{4^{2+\varepsilon}CC'C''}{k^{2+\varepsilon}}.
  \end{equation}
  We have again arrived at condition $(i)$.

  \bigskip
  An analogous argument can be used if we assume condition $(iii)$, that is, when
  \begin{equation}\label{}
    \max_{j=1,\ldots,m} \| T^{k} P_{M_j^\perp}\| \leq \frac{C'}{k^{1/2+\varepsilon}}
  \end{equation}
  for some $\varepsilon>0$ and some $C'>0$. Then, instead of \eqref{pr:thresholdsEps:ineq2}, we use
  \begin{align}\label{}\nonumber
    \max_{i,j=1,\ldots,m} \|P_{M_i^\perp} T^{2n} P_{M_j^\perp}\|
    & \leq \max_{i=1,\ldots,m} \|P_{M_i^\perp} T^{n}\| \cdot \max_{j=1,\ldots,m} \|T^{n} P_{M_j^\perp}\| \\
    & \leq \frac{C'}{n^{1/2+\varepsilon}} \max_{i=1,\ldots,m} \|P_{M_i^\perp} T^{n} \|
  \end{align}
  combined with \eqref{pr:thresholdsEps:ineq1} and Lemma \ref{lem:thresholds} (ii).

  \bigskip
  ``$(iv) \Rightarrow (v). $'' Assume that $(iv)$ holds. Then, in view of Lemma \ref{thm:normsInPS}, we also obtain condition $(ix)$. However, inequality \eqref{pr:thresholds:x:ToShow} together with  Lemma \ref{lem:thresholds} (viii)  lead us to condition $(x)$ with $\varepsilon' := \varepsilon/2 >0$. Again, thanks to Lemma \ref{thm:normsInPS} we obtain condition $(v)$ with $\varepsilon'>0$, which completes the proof.
\end{proof}

We now arrive at the main result of our paper.

\begin{theorem}\label{thm:main}
   For each $y_0 \in \mathcal H$,  the sequence $\{y_k\}_{k=0}^\infty$ defined by \eqref{int:yk} satisfies
  \begin{equation}\label{thm:main:rateAR1}
    \|y_k-y_{k-1}\| =  o (k^{-1})
  \end{equation}
  and
  \begin{equation}\label{thm:main:rateAVD1}
    {\textstyle\sqrt{\frac 1 m \sum_{i=1}^m d^2(y_k,M_i)}} =  o (k^{-1/2}).
  \end{equation}
  Moreover, for each $y_0 \in  Y:= M \oplus \sum_{i=1}^{m}M_i^\perp$, the sequence $\{y_k\}_{k=0}^\infty$ defined by \eqref{int:yk} satisfies
  \begin{equation} \label{thm:main:rateYk}
    \|y_k  - P_M(y_0)  \| = \mathcal O(k^{-1/2}),
  \end{equation}
  \begin{equation}\label{thm:main:rateAR2}
    \|y_k-y_{k-1}\| =  \mathcal O(k^{-3/2})
  \end{equation}
  and
  \begin{equation}\label{thm:main:rateAVD2}
    {\textstyle\sqrt{\frac 1 m \sum_{i=1}^m d^2(y_k,M_i)}} =  \mathcal O(k^{-1}).
  \end{equation}
  Furthermore,  if $\sum_{i=1}^{m}M_i^\perp$ is not closed, then  all of the above-mentioned rates  \eqref{thm:main:rateAR1}--\eqref{thm:main:rateAVD2}  are best possible as the corresponding polynomials $k^{1/2}, k$ and $k^{3/2}$ cannot be replaced by $k^{1/2+\varepsilon}, k^{1+\varepsilon}$ and $k^{3/2+\varepsilon}$, respectively, for any $\varepsilon > 0$.
\end{theorem}

\begin{proof}
  Let $y_0 \in \mathcal H$. The statement \eqref{thm:main:rateAR1} is a repetition of Theorem \ref{thm:ARMi} while \eqref{thm:main:rateAVD1} follows from Lemma \ref{lem:convSets} (see \eqref{lem:convSets:ineq2}).

  We now proceed to proving \eqref{thm:main:rateYk}--\eqref{thm:main:rateAVD2}, all of which follow from Lemma \ref{lem:thresholds}. To this end, assume that  $y_0 \in Y$, say $y_0 = x_0 + \frac 1 m \sum_{i=1}^{m}x_i$, where $x_0 \in M$ and  $x_i \in M_i^\perp$. Moreover, let $\mathbf x := (x_1,\ldots,x_m) \in \mathbf H$ and let $\mathbf T$ be defined by \eqref{def:Tbold}.  Recall that $P_{M_i}P_M = P_M P_{M_i} = P_M$ and so $P_{M_i^\perp}P_M = P_M P_{M_i^\perp} = 0$ (use, for example, \cite[Lemma 9.2]{Deutsch2001}).  In particular, using the identities $x_0 = P_M(x_0)$ and $x_i = P_{M_i^\perp}(x_i)$, we have $T(x_0) = x_0$ and $P_M(x_i) = 0$. Then, by \eqref{eq:PCPMperp} and by Lemma \ref{lem:thresholds}, we get
  \begin{align}\label{}
    \|y_k  - P_M(y_0)\|
    & =  \left\|T^k \left(\frac 1 m \sum_{i=1}^{m} P_{M_i^\perp}(x_i)  \right) \right\|
    = \|\mathbf T^k P_{\mathbf D} P_{\mathbf C^\perp}(\mathbf x)\|
    = \mathcal O(k^{-1/2}),
  \end{align}
  \begin{align}\label{} \nonumber
    \|y_k - y_{k-1}\|
     & =  \left\|(T^k - T^{k-1})\left(\frac 1 m \sum_{i=1}^{m} P_{M_i^\perp}(x_i)  \right) \right\|\\
     & = \|(\mathbf T^k - \mathbf T^{k-1})P_{\mathbf D} P_{\mathbf C^\perp}(\mathbf x)\|
     = \mathcal O(k^{-3/2})
  \end{align}
  and,  since  $d(y_k, M_i) = \|P_{M_i^\perp}(y_k)\|$, we also get
  \begin{align}\label{} \nonumber
    \sqrt{\frac 1 m \sum_{i=1}^m d^2(y_k,M_i)}
    & = \sqrt{\frac 1 m \sum_{i=1}^m \left\|P_{M_i^\perp}T^k \left(\frac 1 m \sum_{i=1}^{m} P_{M_i^\perp}(x_i) \right) \right\|^2}  \\
    & = \|P_{\mathbf C^\perp}\mathbf T^k P_{\mathbf D} P_{\mathbf C^\perp}(\mathbf x) \|
    = O(k^{-1}).
  \end{align}

   The fact that \eqref{thm:main:rateAR1}--\eqref{thm:main:rateAVD2} cannot be improved follows directly from Lemma \ref{lem:thresholdsEps} and the uniform boundedness principle \cite[Theorem 2.2]{Brezis2011}. For the convenience of the reader we sketch the proof for  the average distance in \eqref{thm:main:rateAVD1} and in \eqref{thm:main:rateAVD2}.

  To this end, let $\varepsilon >0$ and suppose to the contrary that
  \begin{equation}\label{}
    {\textstyle\sqrt{\frac 1 m \sum_{i=1}^m d^2(y_k,M_i)}} =  o (k^{-1/2-\varepsilon})
  \end{equation}
  for all $y_0 \in \mathcal H$. Then, for each $\mathbf x = (x_1,\ldots,x_m) \in \mathbf H$ and for $y_0 := \frac 1 m \sum_{i=1}^{m}x_i$, we get
  \begin{equation}\label{}
    \sup_{k=1,2,\ldots}k^{1/2+\varepsilon} \|P_{\mathbf C^\perp}\mathbf T^k P_{\mathbf D}(\mathbf x) \|
    = \sup_{k=1,2,\ldots}k^{1/2+\varepsilon} {\textstyle\sqrt{\frac 1 m \sum_{i=1}^m d^2(y_k,M_i)}}  < \infty.
  \end{equation}
  By the uniform boundedness principle \cite[Theorem 2.2]{Brezis2011} applied to the family of operators $\{k^{1/2+\varepsilon} P_{\mathbf C^\perp}\mathbf T^k P_{\mathbf D} \colon  k=1,2,\ldots \}$, we obtain
  \begin{equation}\label{}
    \sup_{k=1,2,\ldots}k^{1/2+\varepsilon} \|P_{\mathbf C^\perp} \mathbf T^k P_{\mathbf D}\| < \infty,
  \end{equation}
  which corresponds to condition (vii)  in Lemma \ref{lem:thresholdsEps}. This implies that $\sum_{i=1}^{m}M_i^\perp$ is closed, which is in contradiction with our assumption.

  A similar argument can be used when we assume that
  \begin{equation}\label{}
    {\textstyle\sqrt{\frac 1 m \sum_{i=1}^m d^2(y_k,M_i)}} =  \mathcal O (k^{-1-\varepsilon})
  \end{equation}
  for all $y_0 \in  Y $. Indeed, for each $\mathbf x = (x_1,\ldots,x_m) \in  \mathbf H$ and  since $y_0 := \frac 1 m\sum_{i=1}^m  P_{M_i^\perp}(x_i) \in Y$,  we get
  \begin{equation}\label{}
    \sup_{k=1,2,\ldots}k^{1+\varepsilon} \| P_{\mathbf C^\perp} \mathbf T^k P_{\mathbf D}P_{\mathbf C^\perp}(\mathbf x) \|
    = \sup_{k=1,2,\ldots}k^{1+\varepsilon} {\textstyle\sqrt{\frac 1 m \sum_{i=1}^m d^2(y_k,M_i)}}  < \infty.
  \end{equation}
  Again, by using the uniform boundedness principle \cite[Theorem 2.2]{Brezis2011}, but this time applied to the family of operators $\{k^{1+\varepsilon}P_{\mathbf C^\perp} \mathbf T^k P_{\mathbf D} P_{\mathbf C^\perp}  \colon  k=1,2,\ldots \}$, we obtain
  \begin{equation}\label{}
     \sup_{k=1,2,\ldots}k^{1+\varepsilon} \|P_{\mathbf C^\perp}\mathbf T^k P_{\mathbf D}P_{\mathbf C^\perp}\| < \infty,
  \end{equation}
  which corresponds to condition (x) in Lemma \ref{lem:thresholdsEps}. This again leads to contradiction with our assumption.

  Analogously, we can show that  if
  \begin{equation}\label{}
    \|y_k - y_{k-1}\| = o(k^{-1-\varepsilon})
  \end{equation}
  holds for all $y_0 \in \mathcal H$, then we arrive at condition (i) (or (vi)) of Lemma \ref{lem:thresholdsEps}.   Furthermore, if any of the  conditions
  \begin{equation}\label{}
    \|y_k - P_M(y_0)\| = \mathcal O(k^{-1/2-\varepsilon}) \quad \text{or} \quad \|y_k - y_{k-1}\| = \mathcal O(k^{-3/2-\varepsilon})
  \end{equation}
  holds for all $ y_0 \in Y$, then we obtain conditions  (viii) or (ix)  from Lemma \ref{lem:thresholdsEps}, respectively.
\end{proof}

\begin{corollary} \label{cor:Xp}
  Assume that $\sum_{i=1}^{m}M_i^\perp$ is not closed. Let $X_p$ be defined as in Theorem \ref{thm:superPoly}, $p = 1,2,\ldots,$ and let $Y$ be defined as in Theorem \ref{thm:main}. Then the polynomial $k^{p}$ cannot be replaced by $k^{p+\varepsilon}$ in \eqref{thm:superPoly:eq} for any $\varepsilon >0$. In particular, the inclusions $X_{p+1} \subset X_p \subset Y$ are strict.
\end{corollary}

\begin{proof}
  In order to show that the rate in \eqref{thm:superPoly:eq} cannot be improved we use  an induction  argument with respect to $p$.

  Suppose first that
  \begin{equation}\label{pr:Xp:firstStep}
    \|y_k - P_M(y_0)\| = o(k^{-1-\varepsilon})
  \end{equation}
  holds for all $y_0 \in X_1$ and some $\varepsilon>0$. Then, for all $y_0' \in \mathcal H$ with $y_k' := T^k(y_0')$, we obtain
  \begin{equation}\label{}
    \|y_k' - y_{k+1}'\| = \|T^k(y_0'-y_1')\| =o(k^{-1-\varepsilon'})
  \end{equation}
  as $y_0'-y_1' \in X_1$ and $P_M(y_0'-y_1') = 0$. This, however, contradicts Theorem \ref{thm:main} in view of which the rate in \eqref{thm:main:rateAR1} cannot be improved.

  Suppose now that
  \begin{equation}\label{pr:Xp:assumption}
    \|y_k - P_M(y_0)\| = o(k^{-p-\varepsilon})
  \end{equation}
  holds for all $y_0 \in X_p$ and some $\varepsilon > 0$, where $p \geq 2$. We show that an analogous relation holds for $p-1$  with $\varepsilon/2$.  Indeed, let $y_0 \in X_{p-1}$,  say $y_0 = x + (I-T)^{p-1}(y)$, where $x \in M$ and $y \in \mathcal H$.  Then, for each $n > k$, we get
  \begin{align}\label{pr:Xp:estimate1}
    \|y_k - P_M(y_0)\|  \leq \sum_{i=k}^{n}\|y_i-y_{i+1}\| + \|y_{n+1}-P_M(y_0)\|.
  \end{align}
  Note that
  \begin{equation}\label{}
    y_i - y_{i+1} = T^i(I-T)(y_0) = T^i(I-T)^p(y),
  \end{equation}
  where $(I-T)^p(y) \in X_p$. Moreover, because of our assumption (see \eqref{pr:Xp:assumption}) combined with the uniform boundeedness principle \cite[Theorem 2.2]{Brezis2011},
  \begin{equation}\label{}
    C := \sup_{k=1,2,\ldots} k^{p+\varepsilon} \|T^k(I-T)^p\| < \infty.
  \end{equation}
  Thus,  by letting $n \to \infty$ in \eqref{pr:Xp:estimate1}, we obtain
  \begin{align}\label{} \nonumber
    \|y_k - P_M(y_0)\| & \leq \sum_{i=k}^{\infty}\|T^i(I-T)^p(y)\|
    \leq \sum_{i=k}^{\infty}\frac C {i^{p+\varepsilon}} \\
    & \leq \int_{i=k-1}^{\infty} \frac C {x^{p+\varepsilon}}dx = \frac { C(p-1+\varepsilon)^{-1}} {(k-1 )^{p-1+\varepsilon}}
  \end{align}
  In particular, for all $y_0 \in X_{p-1}$, we get
  \begin{equation}\label{pr:Xp:toShow}
    \|y_k - P_M(y_0)\| = o(k^{-(p-1)-\varepsilon/2}),
  \end{equation}
  as claimed.

  By repeating the above-mentioned argument, we arrive at \eqref{pr:Xp:firstStep} with some $\varepsilon'>0$. Consequently, we have shown that the rate in \eqref{thm:superPoly:eq} cannot be improved.

  Observe that the latter statement implies that the subspaces $X_p$ are distinct for different values of $p = 1,2,\ldots$. Indeed, if we suppose otherwise, that $X_p = X_{p+1}$ for some $p \geq 1$, then this would imply \eqref{pr:Xp:assumption} with $\varepsilon = 1$. However, as we have shown above, this situation cannot happen.

  Similarly, if $X_p = Y$ for some $p\geq 1$, then this would imply that $\|y_k-P_M(y_0)\| = o(k^{-1/2 - \varepsilon})$ for all $y_0 \in Y$, where $\varepsilon = p - 1/2$. This however would contradict Theorem \ref{thm:main} in view of which the rate in \eqref{thm:main:rateYk} cannot be improved. We note here that the inclusion $X_1 \subset Y$ can be easily deduced from \eqref{pr:thresholdsEps:I-T}.
\end{proof}

The following result provides an alternative explanation for the fact that the rates of \eqref{thm:main:rateAR1}--\eqref{thm:main:rateAVD2} cannot be improved.

\begin{theorem} \label{thm:main2}
  Assume that $\sum_{i=1}^{m} M_i^\perp$ is not closed. Then there is a dense subset $U$ of $\mathcal H$ such that for each $y_0 \in U$ the sequence $\{y_k\}_{k=0}^\infty$ defined in \eqref{int:yk} satisfies
  \begin{equation}\label{thm:main2:rateAR}
    \limsup_{k \to \infty} k^{1+\varepsilon}\|y_k-y_{k-1}\| = \infty
  \end{equation}
  and
  \begin{equation}\label{thm:main2:rateMaxAVD1}
    \limsup_{k \to \infty} k^{1/2+\varepsilon} {\textstyle\sqrt{\frac 1 m \sum_{i=1}^m d^2(y_k,M_i)}} = \infty
  \end{equation}
  for all $\varepsilon >0$. Moreover, there is a dense subset $V$ of $Y = M \oplus \sum_{i=1}^{m}M_i^\perp$ such that for each $y_0 \in V$ the sequence $\{y_k\}_{k=0}^\infty$ defined in \eqref{int:yk} satisfies
  \begin{equation}\label{thm:main2:rateYk}
    \limsup_{k \to \infty} k^{1/2+\varepsilon} \|y_k-P_M(y_0)\| = \infty,
  \end{equation}
  \begin{equation}\label{thm:main2:rateAR2}
    \limsup_{k \to \infty} k^{3/2+\varepsilon}\|y_k-y_{k-1}\| = \infty
  \end{equation}
  and
  \begin{equation}\label{thm:main2:rateMaxAVD2}
    \limsup_{k \to \infty} k^{1+\varepsilon} {\textstyle\sqrt{\frac 1 m \sum_{i=1}^m d^2(y_k,M_i)}} = \infty
  \end{equation}
  for all $\varepsilon >0$.
\end{theorem}

\begin{proof}
  We first define the subset $V$ and then show equalities \eqref{thm:main2:rateYk}--\eqref{thm:main2:rateMaxAVD2}. To this end, let $\{\varepsilon_n\}_{n=0}^\infty \subset (0,\infty)$ be such that $\varepsilon_n \downarrow 0$. By Lemma \ref{lem:thresholdsEps}, for each $n = 1,2,\ldots$, we have
  \begin{equation}\label{}
    \sup_{k=1,2,\ldots} k^{1+\varepsilon_n} \|P_{\mathbf C^\perp} \mathbf T^k P_{\mathbf D} P_{\mathbf C^\perp}\| = \infty.
  \end{equation}
  Consequently, by  applying  the strong contrapositive of the  uniform boundedness principle \cite[Theorem 5.4.10]{Simon2015}  (and the successive Remark on p. 399 in \cite{Simon2015}) to the family of operators $\{ k^{1/2+\varepsilon_n} P_{\mathbf C^\perp}\mathbf T^k P_{\mathbf D}P_{\mathbf C^\perp} \colon  k=1,2,\ldots \}$, we see that
  \begin{equation}\label{}
    \mathbf V_n := \left \{ \mathbf v \colon \sup_{k=1,2,\ldots} k^{1+\varepsilon_n} \|P_{\mathbf C^\perp} \mathbf T^k P_{\mathbf D} P_{\mathbf C^\perp}(\mathbf v)\| = \infty\right\}
  \end{equation}
  is a dense $G_\delta$ subset of $\mathbf H$. By the Baire category theorem \cite[Theorem 5.4.1]{Simon2015}, the subset $\mathbf V := \bigcap_{n=0}^\infty \mathbf V_n$ is also a dense $G_\delta$ subset of $\mathbf H$. In fact, we have
  \begin{equation}\label{pr:main2:boldV}
    \mathbf V = \left \{ \mathbf v \colon \limsup_{k \to \infty} k^{1+\varepsilon} \|P_{\mathbf C^\perp} \mathbf T^k P_{\mathbf D} P_{\mathbf C^\perp}(\mathbf v)\| = \infty \text{ for all } \varepsilon >0\right\}.
  \end{equation}
  We may now define the aforementioned subset $V$ in $\mathcal H$ by
  \begin{equation}\label{}
    V := \left\{ v = v_0 + \frac 1 m \sum_{i=1}^{m} P_{M_i^\perp} (v_i) \colon v_0 \in M \text{ and } \mathbf v = (v_1, \ldots, v_m) \in \mathbf V \right\} \subset Y.
  \end{equation}

  We show that $V$ is dense in $Y$, that is, for each $y \in Y$, there is a sequence $\{v_k\}_{k=0}^\infty \subset V$ such that $v_k \to y$. To this end, suppose that $y = x_0 + \frac 1 m \sum_{i=1}^{m} x_i$, where $x_0 \in M$ and where $x_i \in M_i^\perp$. Moreover, let $\mathbf x := (x_1,\ldots,x_m)$. Since $\mathbf V$ is a dense subset of $\mathbf H$, there is a sequence $\{\mathbf v_k\}_{k=0}^\infty \subset \mathbf V$, with $\mathbf v_k = (v_{k,1},\ldots, v_{k,m})$, satisfying $\mathbf v_k \to \mathbf x$ as $k \to \infty$.
  Equivalently, $v_{k,i} \to x_i$ as $k \to \infty$ for all $i = 1,\ldots,m$. In particular, for $ v_k := x_0 + \frac 1 m \sum_{i=1}^{m} P_{M_i^\perp}(v_{k,i}) \in V$, we get
  \begin{equation}\label{}
    \|v_k - y\| = \left\|\frac 1 m \sum_{i=1}^{m}P_{M_i^\perp}(v_{k,i}-x_i) \right\| \leq \frac 1 m \sum_{i=1}^{m}\|v_{k,i}-x_i\| \to 0
  \end{equation}
  as $k \to \infty$. Since $Y$ is a dense subset of $\mathcal H$, we have also established that $V$ is a dense in $\mathcal H$.

  We may now turn our attention to equalities \eqref{thm:main2:rateYk}--\eqref{thm:main2:rateMaxAVD2}. Note that for each $y_0 \in V$, say $y_0 = v_0 + \frac 1 m \sum_{i=1}^{m}P_{M_i^\perp}(v_i)$, we have
  \begin{equation}\label{}
    {\textstyle\sqrt{\frac 1 m \sum_{i=1}^m d^2(y_k,M_i)}}
    = \|P_{\mathbf C^\perp} \mathbf T^k P_{\mathbf D} P_{\mathbf C^\perp}(\mathbf v)\|,
  \end{equation}
  where $\mathbf v = (v_1,\ldots,v_m) \in \mathbf V$. This, when combined with \eqref{pr:main2:boldV}, shows \eqref{thm:main2:rateMaxAVD2}.
  On the other hand, by Theorem \ref{thm:main}, we see that
  \begin{equation}\label{pr:main2:yk}
    \|y_k - P_M(y_0)\| \leq \frac C {k^{1/2}}
    \quad \text{and} \quad
    \|y_k-y_{k-1}\| \leq \frac C {k^{3/2}}
  \end{equation}
  for some $C>0$, where $k = 1,2,\ldots$. Thus, by Lemma \ref{lem:convSets} (see \eqref{lem:convSets:ineq1}), we arrive at
  \begin{align}\label{pr:main2:lowerbound} \nonumber
    k^{2+\varepsilon} { \frac 1 m \sum_{i=1}^{m} d^2(y_k, M_i)}
    & \leq k^{2+\varepsilon} \frac m 2 \|y_k-y_{k-1}\| \|y_k - P_M(y_0)\| \\
    & \leq
    \begin{cases}\displaystyle
      \frac{Cm}{2 } k^{3/2 + \varepsilon} \|y_k-y_{k-1}\|\\ ~\\
      \displaystyle
      \frac{Cm}{2 } k^{1/2 + \varepsilon}\|y_k - P_M(y_0)\|.
    \end{cases}
  \end{align}
  After taking the $\limsup$ as $k \to \infty$ in \eqref{pr:main2:lowerbound}, and using \eqref{thm:main2:rateMaxAVD2}, we arrive at \eqref{thm:main2:rateYk} and \eqref{thm:main2:rateAR2}.

  \bigskip
  A similar argument can be used in order to define a dense subset $U$ on which equalities \eqref{thm:main2:rateAR} and \eqref{thm:main2:rateMaxAVD1} hold. Indeed, by Lemma \ref{lem:thresholdsEps}, we have
  \begin{equation}\label{}
    \sup_{k=1,2,\ldots} k^{1/2+\varepsilon_n} \|P_{\mathbf C^\perp} \mathbf T^k P_{\mathbf D}\| = \infty,
  \end{equation}
  where $\{\varepsilon_n\}_{n=0}^\infty$ is as above.  By using the above-mentioned strong contrapositive of the uniform boundedness principle, but this time applied to the family of operators  $\{ k^{1/2+\varepsilon_n}$ $P_{\mathbf C^\perp}\mathbf T^k P_{\mathbf D} \colon k=1,2,\ldots \}$, we obtain that
  \begin{equation}\label{}
    \mathbf U_n := \left\{ \mathbf u \colon \sup_{k=1,2,\ldots} k^{1/2+\varepsilon_n} \|P_{\mathbf C^\perp} \mathbf T^k P_{\mathbf D} (\mathbf u)\| = \infty \right\}
  \end{equation}
  is a dense $G_\delta$ subset of $\mathbf H$. By again invoking the Baire category theorem, the set $\mathbf U := \bigcap_{n=0}^\infty \mathbf U_n$ is a dense $G_\delta$ subset of $\mathbf H$. Since $\mathbf U$ satisfies
  \begin{equation}\label{}
    \mathbf U = \left\{ \mathbf u \colon \limsup_{k \to \infty} k^{1/2+\varepsilon} \|P_{\mathbf C^\perp} \mathbf T^k P_{\mathbf D} (\mathbf u)\| = \infty \text{ for all } \varepsilon >0\right\},
  \end{equation}
  it suffices to put
  \begin{equation}\label{}
    U := \left\{ u = u_0 + \frac 1 m \sum_{i=1}^{m} u_i \colon u_0 \in M \text{ and } \mathbf u = (u_1, \ldots, u_m) \in \mathbf U \right\}.
  \end{equation}
  It is not difficult to see that $U$ is a dense subset of $\mathcal H$ (because $\mathbf U$ is dense in $\mathbf H$). Moreover, for each $y_0 = u_0 + \frac 1 m \sum_{i=1}^{m} u_i \in U $, we have
  \begin{equation}\label{}
    k^{1+\varepsilon} \|P_{\mathbf C^\perp} \mathbf T^k P_{\mathbf D}(\mathbf u)\|^2
    = k^{1+\varepsilon} { \frac 1 m \sum_{i=1}^{m} d^2(y_k, M_i)}
    \leq k^{1+\varepsilon} \frac m 2 \|y_k-y_{k-1}\| \|y_0\|,
  \end{equation}
  where $\mathbf u = (u_1,\ldots,u_m) \in \mathbf U$. By taking the $\limsup$ as $k \to \infty$ we arrive at \eqref{thm:main2:rateAR} and \eqref{thm:main2:rateMaxAVD1}.
\end{proof}

\begin{remark}
  Thanks to the inequalities
  \begin{equation}\label{}
    {\textstyle\sqrt{\frac 1 m \sum_{i=1}^m d^2(y_k, C_i)}}
    \leq \max_{i=1,\ldots,m} d(y_k, M_i)
    \leq {\textstyle\sqrt{\sum_{i=1}^m d^2(y_k, C_i)}},
  \end{equation}
  which hold for all $y_0 \in \mathcal H$, we may equivalently replace the average distance by the maximum distance in Theorems \ref{thm:main} and \ref{thm:main2}.
\end{remark}

\begin{remark}[Lower Bounds] \label{rem:lowerBounds}
  Theorem \ref{thm:main2} implies that for each $y_0 \in U$, $\varepsilon >0$ and $C >0$, the lower bounds
  \begin{equation}\label{}
    \|y_k - y_{k-1}\| \geq \frac {C}{k^{1+\varepsilon}} \quad \text{and} \quad
    {\textstyle\sqrt{\frac 1 m \sum_{i=1}^m d^2(y_k, C_i)}} \geq \frac {C}{k^{1/2+\varepsilon}}
  \end{equation}
  hold for infinitely many $k's$. Similarly, for each $y_0 \in V$, $y_0 \in U$, $\varepsilon >0$ and $C >0$, the lower bounds
  \begin{equation}\label{}
    \|y_k - P_M(y_0)\| \geq \frac {C}{k^{1/2+\varepsilon}}, \quad
    \|y_k - y_{k-1}\| \geq \frac {C}{k^{3/2+\varepsilon}}
  \end{equation}
  and
  \begin{equation}\label{}
    {\textstyle\sqrt{\frac 1 m \sum_{i=1}^m d^2(y_k, C_i)}} \geq \frac {C}{k^{1+\varepsilon}}
  \end{equation}
  hold for infinitely many $k's$. This corresponds to \eqref{int:lowerBound}. We do not know whether it is possible to show that the above-mentioned lower bounds hold for all sufficiently large $k$'s. Equivalently, we do not know if $\limsup$ of Theorem \ref{thm:main2} can be replaced by $\liminf$. We leave this as an open problem.
\end{remark}

\begin{remark}
  We have become aware of a paper by Evron et al. \cite{Evronetal2022} in which the method of cyclic projections is studied in the context of machine learning. Although formulated in a different setting and employing different assumptions in its analysis (in particular, $\mathcal H = \mathbb R^d$), this paper is related to the results presented here.
\end{remark}

\section*{Appendix} \label{sec:Appendix}
\addcontentsline{toc}{section}{Appendix}
 In this section we sketch how to derive Theorems \ref{thm:ARMi} and \ref{thm:superPoly} in a real Hilbert space, having in mind that the corresponding results of \cite{BadeaSeifert2016} were established in a complex Hilbert space. We also present an alternative proof of Theorem \ref{thm:ARMi} by using \cite[Lemma 5.2]{Crouzeix2008}. For this purpose, we use a complexification argument. For more details concerning the complexification, we refer the reader to \cite{LunaRamirezShapiro2012}.

\bigskip
To this end, let $ \mathbf H_{\mathbb C}  := \mathcal H + i\mathcal H$ be the  (external)  complexification of $\mathcal H$ with scalar multiplication given by
\begin{equation}\label{}
  (\alpha + i\beta)(x+iy) := \alpha x - \beta y +i (\alpha y + \beta x)
\end{equation}
and inner product $\langle \cdot, \cdot \rangle_{\mathbb C}$ defined by
\begin{equation}\label{complex:innerProd}
  \langle x+iy, x'+iy'\rangle_{\mathbb C} := \langle x,x'\rangle + \langle y,y'\rangle
  + i(\langle x',y\rangle - \langle x,y'\rangle),
\end{equation}
where $\alpha,\beta \in \mathbb R $ and $x,y,x',y'\in \mathcal H$. Thus, the induced norm on $ \mathbf H_{\mathbb C} $, denoted by $\|\cdot\|_{\mathbb C}$, satisfies
\begin{equation}\label{complex:innerNorm}
  \|x+iy\|_{\mathbb C}^2 = \|x\|^2 + \|y\|^2
\end{equation}
for all $x+iy\in  \mathbf H_{\mathbb C} $. It is not difficult to see that $( \mathbf H_{\mathbb C} , \langle \cdot, \cdot \rangle_{\mathbb C})$ is indeed a complex Hilbert space.

\begin{proof}[Proof of Theorem \ref{thm:ARMi}]
For each $j = 1, \ldots, m$, let $\mathbf M_j  := M_j + i M_j$. Observe that $\mathbf M_j $ is a closed linear subspace of $ \mathbf H_{\mathbb C} $. Denote by $P_{\mathbf M_j }$ the orthogonal projection onto $\mathbf M_j $. Then, for each $ \mathbf z  = x+iy \in  \mathbf H_{\mathbb C} $, we have $ P_{\mathbf M_j}(\mathbf z)  = P_{M_j}(x) + i P_{M_j}(y)$. This implies that the product $\mathbf T := P_{\mathbf M_m} \ldots P_{\mathbf M_1} $ satisfies $ \mathbf T( \mathbf z)  = T(x) +iT(y)$, where $T$ is defined as in \eqref{int:T}.  Using induction,  we get
\begin{equation}\label{pr:ARMi:T}
    \mathbf T^k(\mathbf z) = T^k(x) + i T^k(y).
\end{equation}
By \cite[Remark 4.2(b) and Theorem 2.1]{BadeaSeifert2016}, we  have
\begin{equation}\label{}
  k \|\mathbf T^k(\mathbf z) - \mathbf T^{k-1}(\mathbf z) \|_{\mathbb C} \to 0 \quad \text{as } k \to \infty
\end{equation}
for all $ \mathbf z \in  \mathbf H_{\mathbb C} $. In particular, by taking $ \mathbf z :=  y_0 +i0$, we obtain
  \begin{equation}\label{}
    \|y_k - y_{k-1}\| = \|\mathbf T^k(\mathbf z) - \mathbf T^{k-1}(\mathbf z) \|_{\mathbb C}.
  \end{equation}
This implies Theorem \ref{thm:ARMi}.
\end{proof}

\begin{proof}[Alternative proof of Theorem \ref{thm:ARMi}]
  By \cite[Lemma 5.2]{Crouzeix2008} applied to the operator $ \mathbf T$, we have
  \begin{equation}\label{}
    \sum_{k=1}^{\infty} k \|\mathbf T^k(\mathbf z) - \mathbf T^{k-1}(\mathbf z) \|_{\mathbb C}^2 < \infty
  \end{equation}
for all $ \mathbf z \in \mathbf H_{\mathbb C}$. In particular, by taking $ \mathbf z :=  y_0 +i0$, and knowing that  the sequence  $\{\|y_k - y_{k-1}\|\}_{k=1}^\infty$ is decreasing, we have
  \begin{align}\label{} \nonumber
    k^2 \|y_k - y_{k-1}\|^2
    & \leq 2 k \lceil k/2 \rceil\|y_k - y_{k-1}\|^2
    \leq  4  \sum^{k}_{n= \lfloor k/2\rfloor + 1} \frac k 2 \|y_n - y_{n-1}\|^2 \\
    & \leq  4  \sum^{k}_{n= \lfloor k/2\rfloor + 1} n \|y_n - y_{n-1}\|^2
    =  4  \sum^{{k}}_{n= \lfloor k/2\rfloor + 1} n \|\mathbf T^n(\mathbf z) - \mathbf T^{n-1}(\mathbf z) \|^2_{\mathbb C} \to 0
  \end{align}
as $k \to \infty$.  Thus we have shown that $\|y_k-y_{k-1}\| = o(k^{-1})$ for all $y_0\in\mathcal H$.
\end{proof}

\begin{proof}[Proof of Theorem \ref{thm:superPoly}]
  It is not difficult to see that $\mathbf M_j^\perp = M_j^\perp + i M_j^\perp$ for all $j = 1,\ldots,m$; compare with \eqref{eq:Cperp}. Consequently,
  \begin{equation}\label{} \textstyle
    \sum_{j=1}^{m}\mathbf M_j^\perp = \sum_{j=1}^{m}M_j^\perp + i \sum_{j=1}^{m}M_j^\perp
  \end{equation}
  and thus
  \begin{equation}\label{} \textstyle
    \sum_{j=1}^{m} M_j^\perp \text{ is (not) closed} \Longleftrightarrow \sum_{j=1}^{m}\mathbf M_j^\perp \text{ is (not) closed}.
  \end{equation}

  Consider now the subspaces of $\mathbf H_{\mathbb C}$ given by $\mathbf X_p := \mathbf M \oplus (\mathbf I - \mathbf T)^p(\mathbf H_{\mathbb C})$ and $\mathbf X := \textstyle \bigcap_{p=1}^\infty \mathbf X_p$, where $\mathbf M := M+iM$ and $\mathbf I$ is the identity operator on $\mathbf H_{\mathbb C}$, $p=1,2,\ldots$. Obviously, $\mathbf X_p$ and $\mathbf X$ are the analogues of $X_p$ and $X$ considered in Theorem \ref{thm:superPoly}. In fact, we have
   \begin{equation}\label{}
    \mathbf X_p = X_p + i X_p \quad \text{and} \quad \mathbf X = X + i X.
  \end{equation}
  Consequently, we obtain
  \begin{equation}\label{}
    X_p \ (X) \text{ is dense in } \mathcal H \Longleftrightarrow \mathbf X_p \ ( \mathbf X) \text{ is dense in } \mathbf H_{\mathbb C}.
  \end{equation}
  By \cite[Theorem 4.3]{BadeaSeifert2016}, for each $\mathbf z \in \mathbf X_p$, we get
  \begin{equation}\label{}
    \|\mathbf T^k(\mathbf z) - P_{\mathbf M}(\mathbf z)\| = o(k^{-p})
  \end{equation}
  where $p = 1,2,\ldots$. Thus, for each $y_0 \in X_p$ it suffices to take $\mathbf z := y_0 +i0 \in \mathbf X_p$, to see that
  \begin{equation}\label{pr:superPoly:ykOnXp}
    \|T^k(y_0) - P_M(y_0)\| = \|\mathbf T^k(\mathbf z) - P_{\mathbf M}(\mathbf z)\| = o(k^{-p}),
  \end{equation}
  which shows \eqref{thm:superPoly:eq}. Similarly, for each $y_0 \in X$, it suffices to take $\mathbf z := y_0 +i0 \in \mathbf X$ to see that \eqref{pr:superPoly:ykOnXp} holds for all $p > 0$. Moreover, by \cite[Theorem 4.3]{BadeaSeifert2016}, we know that $\mathbf X_p$ and $\mathbf X$ are dense in $\mathbf H_{\mathbb C}$.
\end{proof}

\section*{Acknowledgement}
\addcontentsline{toc}{section}{Acknowledgement}
Both authors are grateful to two anonymous referees for their pertinent comments and helpful suggestions.

\section*{Funding}
\addcontentsline{toc}{section}{Funding}
This work was partially supported by the Israel Science Foundation (Grants 389/12 and 820/17), the Fund for the Promotion of Research at the Technion and by the Technion General Research Fund.

\section*{Data Availability}
\addcontentsline{toc}{section}{Data Availability}
Data sharing is not applicable to this article as no datasets were generated or analyzed during the current study.

\section*{Conflict of Interest}
\addcontentsline{toc}{section}{Conflict of Interest}
The authors declare that they have no conflict of interest.

\small

\addcontentsline{toc}{section}{References}

\begin{thebibliography}{10}

\bibitem{ArtachoCampoy2019}
{\sc F.~J. Arag\'{o}n~Artacho and R.~Campoy}, {\em Optimal rates of linear
  convergence of the averaged alternating modified reflections method for two
  subspaces}, Numer. Algorithms, 82 (2019), pp.~397--421.

\bibitem{Aronszajn1950}
{\sc N.~Aronszajn}, {\em Theory of reproducing kernels}, Trans. Amer. Math.
  Soc., 68 (1950), pp.~337--404.

\bibitem{BadeaGrivauxMuller2011}
{\sc C.~Badea, S.~Grivaux, and V.~M\"uller}, {\em The rate of convergence in
  the method of alternating projections}, Algebra i Analiz, 23 (2011),
  pp.~1--30.

\bibitem{BadeaSeifert2016}
{\sc C.~Badea and D.~Seifert}, {\em Ritt operators and convergence in the
  method of alternating projections}, J. Approx. Theory, 205 (2016),
  pp.~133--148.

\bibitem{BadeaSeifert2017}
{\sc C.~Badea and D.~Seifert}, {\em Quantified asymptotic behaviour of {B}anach
  space operators and applications to iterative projection methods}, Pure Appl.
  Funct. Anal., 2 (2017), pp.~585--598.

\bibitem{BargetzReichZalas2018}
{\sc C.~Bargetz, S.~Reich, and R.~Zalas}, {\em Convergence properties of
  dynamic string-averaging projection methods in the presence of
  perturbations}, Numer. Algorithms, 77 (2018), pp.~185--209.

\bibitem{BauschkeCruzNghiaPhanWang2014}
{\sc H.~H. Bauschke, J.~Y. Bello~Cruz, T.~T.~A. Nghia, H.~M. Phan, and
  X.~Wang}, {\em The rate of linear convergence of the {D}ouglas-{R}achford
  algorithm for subspaces is the cosine of the {F}riedrichs angle}, J. Approx.
  Theory, 185 (2014), pp.~63--79.

\bibitem{BauschkeCruzNghiaPhanWang2016}
{\sc H.~H. Bauschke, J.~Y. Bello~Cruz, T.~T.~A. Nghia, H.~M. Phan, and
  X.~Wang}, {\em Optimal rates of linear convergence of relaxed alternating
  projections and generalized {D}ouglas-{R}achford methods for two subspaces},
  Numer. Algorithms, 73 (2016), pp.~33--76.

\bibitem{BauschkeBorwein1996}
{\sc H.~H. Bauschke and J.~M. Borwein}, {\em On projection algorithms for
  solving convex feasibility problems}, SIAM Rev., 38 (1996), pp.~367--426.

\bibitem{BauschkeBorweinLewis1997}
{\sc H.~H. Bauschke, J.~M. Borwein, and A.~S. Lewis}, {\em The method of cyclic
  projections for closed convex sets in {H}ilbert space}, in Recent
  developments in optimization theory and nonlinear analysis ({J}erusalem,
  1995), vol.~204 of Contemp. Math., Amer. Math. Soc., Providence, RI, 1997,
  pp.~1--38.

\bibitem{BauschkeDeutschHundal2009}
{\sc H.~H. Bauschke, F.~Deutsch, and H.~Hundal}, {\em Characterizing
  arbitrarily slow convergence in the method of alternating projections}, Int.
  Trans. Oper. Res., 16 (2009), pp.~413--425.

\bibitem{BorodinKopecka2020}
{\sc P.~A. Borodin and E.~Kopeck\'{a}}, {\em Alternating projections, remotest
  projections, and greedy approximation}, J. Approx. Theory, 260 (2020),
  pp.~105486, 16.

\bibitem{Brezis2011}
{\sc H.~Brezis}, {\em Functional analysis, {S}obolev spaces and partial
  differential equations}, Universitext, Springer, New York, 2011.

\bibitem{Cegielski2012}
{\sc A.~Cegielski}, {\em Iterative methods for fixed point problems in
  {H}ilbert spaces}, vol.~2057 of Lecture Notes in Mathematics, Springer,
  Heidelberg, 2012.

\bibitem{Cohen2007}
{\sc G.~Cohen}, {\em Iterates of a product of conditional expectation
  operators}, J. Funct. Anal., 242 (2007), pp.~658--668.

\bibitem{Crouzeix2008}
{\sc M.~Crouzeix}, {\em A functional calculus based on the numerical range:
  applications}, Linear Multilinear Algebra, 56 (2008), pp.~81--103.

\bibitem{Deutsch2001}
{\sc F.~Deutsch}, {\em Best approximation in inner product spaces}, vol.~7 of
  CMS Books in Mathematics, Springer-Verlag, New York, 2001.

\bibitem{DeutschHundal2010}
{\sc F.~Deutsch and H.~Hundal}, {\em Slow convergence of sequences of linear
  operators {II}: arbitrarily slow convergence}, J. Approx. Theory, 162 (2010),
  pp.~1717--1738.

\bibitem{DeutschHundal2011}
{\sc F.~Deutsch and H.~Hundal}, {\em Arbitrarily slow convergence of sequences
  of linear operators: a survey}, in Fixed-point algorithms for inverse
  problems in science and engineering, vol.~49 of Springer Optim. Appl.,
  Springer, New York, 2011, pp.~213--242.

\bibitem{DeutschHundal2015}
{\sc F.~Deutsch and H.~Hundal}, {\em Arbitrarily slow convergence of sequences
  of linear operators}, in Infinite products of operators and their
  applications, vol.~636 of Contemp. Math., Amer. Math. Soc., Providence, RI,
  2015, pp.~93--120.

\bibitem{Evronetal2022} {\sc I.~Evron, E.~Moroshko, R.~Ward, N.~Srebro and D.~Soudry}, {\em How catastrophic can catastrophic forgetting be in linear regression?}, Proceedings of Thirty Fifth Conference on Learning Theory, PMLR 178:4028--4079, (2022).

\bibitem{FranchettiLight1986}
{\sc C.~Franchetti and W.~Light}, {\em On the von {N}eumann alternating
  algorithm in {H}ilbert space}, J. Math. Anal. Appl., 114 (1986),
  pp.~305--314.

\bibitem{Galantai2004}
{\sc A.~Gal\'{a}ntai}, {\em Projectors and projection methods}, vol.~6 of
  Advances in Mathematics (Dordrecht), Kluwer Academic Publishers, Boston, MA,
  2004.

\bibitem{Halperin1962}
{\sc I.~Halperin}, {\em The product of projection operators}, Acta Sci. Math.
  (Szeged), 23 (1962), pp.~96--99.

\bibitem{KayalarWeinert1988}
{\sc S.~Kayalar and H.~L. Weinert}, {\em Error bounds for the method of
  alternating projections}, Math. Control Signals Systems, 1 (1988),
  pp.~43--59.

\bibitem{LunaRamirezShapiro2012}
{\sc M.~E. Luna-Elizarrar\'{a}s, F.~Ram\'{\i}rez-Reyes, and M.~Shapiro}, {\em
  Complexifications of real spaces: general aspects}, Georgian Math. J., 19
  (2012), pp.~259--282.

\bibitem{Neumann1949}
{\sc J.~von Neumann}, {\em On rings of operators. {R}eduction theory}, Ann. of
  Math. (2), 50 (1949), pp.~401--485.

\bibitem{Pierra1984}
{\sc G.~Pierra}, {\em Decomposition through formalization in a product space},
  Math. Programming, 28 (1984), pp.~96--115.

\bibitem{Popa2012}
{\sc C.~Popa}, {\em Projection algorithms - classical results and developments:
  Applications to image reconstruction}, Lambert Academic Publishing, 2012.

\bibitem{ReichZalas2017}
{\sc S.~Reich and R.~Zalas}, {\em The optimal error bound for the method of
  simultaneous projections}, J. Approx. Theory, 223 (2017), pp.~96--107.

\bibitem{ReichZalas2021}
{\sc S.~Reich and R.~Zalas}, {\em Error bounds for the method of simultaneous
  projections with infinitely many subspaces}, J. Approx. Theory, 272 (2021),
  Paper No. 105648, 24 pp.

\bibitem{Simon2015}
{\sc B.~Simon}, {\em Real analysis}, A Comprehensive Course in Analysis, Part
  1, American Mathematical Society, Providence, RI, 2015.
\newblock With a 68 page companion booklet.


\end{thebibliography}

\end{document}